\documentclass[a4paper,11pt,reqno]{amsart}
\usepackage{graphicx}

\usepackage{color}

\usepackage{amsmath,amstext,amssymb,amsopn,amsthm}
\usepackage{url}

\usepackage[margin=30mm]{geometry}
\usepackage{eucal,mathrsfs,dsfont}
\usepackage{soul}

\newtheorem{theorem}{Theorem}[section]
\newtheorem{corollary}[theorem]{Corollary}
\newtheorem{lemma}[theorem]{Lemma}
\newtheorem{proposition}[theorem]{Proposition}

\theoremstyle{definition}

\theoremstyle{remark}

\definecolor{mr}{rgb}{0.1,0.2,0.7}

\newcommand{\eps}{\varepsilon}

\newcommand{\calA}{\mathcal{A}}
\newcommand{\calB}{\mathcal{B}}

\newcommand{\calF}{\mathcal{F}}

\newcommand{\calL}{\mathcal{L}}

\newcommand{\s}{\sigma}

\newcommand{\R}{\mathds{R}}

\newcommand{\N}{{\mathds{N}}}

\newcommand{\Rd}{{\R^d}}

\newcommand{\RR}{\mathrm{I\kern-0.20emR}}
\newcommand{\D}{\mathrm{d}\kern0.2pt}

\newcommand{\E}{\mathbb{E}}
\newcommand{\p}{\mathbb{P}}

\title[Transition densities for SDEs driven by cylindrical stable processes]{Transition density estimates for diagonal systems of SDEs driven by cylindrical $\alpha$-stable processes}
\author[T. Kulczycki]{Tadeusz Kulczycki}
\author[M. Ryznar]{Micha{\l} Ryznar}
\thanks{T. Kulczycki was supported in part by the National Science Centre, Poland, grant no. 2015/17/B/ST1/01233, M. Ryznar was supported in part by the National Science Centre, Poland, grant no. 2015/17/B/ST1/01043}
\address{Faculty of Pure and Applied Mathematics, Wroc{\l}aw University of Science and Technology, Wyb. Wyspia{\'n}skiego 27, 50-370 Wroc{\l}aw, Poland.}
\email{Tadeusz.Kulczycki@pwr.edu.pl}
\email{Michal.Ryznar@pwr.edu.pl}

\begin{document}
\begin{abstract} We consider the system of stochastic differential equation $dX_t = A(X_{t-}) \, dZ_t$, \\$ X_0 = x$,
driven by cylindrical $\alpha$-stable process $Z_t$ in $\R^d$. We assume that $A(x) = (a_{ij}(x))$ is diagonal and $a_{ii}(x)$ are bounded away from zero, from infinity and H{\"o}lder continuous. We construct transition density $p^A(t,x,y)$ of the process $X_t$ and show sharp two-sided estimates of this density. We also prove H{\"o}lder and gradient estimates of $x \to p^A(t,x,y)$. Our approach is based on the method developed by Chen and Zhang in \cite{CZ}.
\end{abstract}

\maketitle

\section{Introduction}
Let 
$$
Z_t = (Z_t^{(1)},\ldots,Z_t^{(d)}),
$$
be cylindrical $\alpha$-stable process, that is $Z_t^{(i)}$, $i = 1,\ldots,d$ are independent one-dimensional symmetric  standard  $\alpha$-stable processes of index $\alpha \in (0,2)$, $d \in \N$, $d \ge 2$. We consider the system of stochastic differential equation 
\begin{equation}
\label{main}
dX_t = A(X_{t-}) \, dZ_t, \quad X_0 = x, 
\end{equation}
where $A(x) = (a_{ij}(x))$ is diagonal and there are constants $b_1, b_2, b_3 > 0$, $\beta \in (0,1]$ such that for any $x, y \in \R^d$, $i \in \{1,\ldots,d\}$
\begin{equation}
\label{bounded}
b_1 \le a_{ii}(x) \le b_2,
\end{equation}
\begin{equation}
\label{aHolder}
|a_{ii}(x) - a_{ii}(y)| \le b_3 |x - y|^{\beta}.
\end{equation}
In the sequel, without loss of generality, we assume that $\beta \in (0,\alpha/4]$.

It is well known that system of SDEs (\ref{main}) has a unique weak solution \cite{BC2006}. The generator of $X$ is given by (see (\cite[(2.3)]{BC2006}))
$$
\calL f(x) =  \sum_{i = 1}^d  \lim_{\varepsilon \to 0^+} \frac{\calA_{\alpha}}{2} \int_{|w_i| > \eps} \left[f(x + a_{ii}(x) w_i e_i) + f(x - a_{ii}(x) w_i e_i) - 2 f(x)\right] \, \frac{dw_i}{|w_i|^{1 + \alpha}},
$$
where $\{e_j\}_{j = 1}^d$ is the standard basis in $\R^d$ and $\calA_{\alpha} = 2^{\alpha} \Gamma((1+\alpha)/2)/(\pi^{1/2} |\Gamma(-\alpha/2)|)$. 

Let us denote the transition density of one-dimensional symmetric standard $\alpha$-stable process of index $\alpha \in (0,2)$ by $g_t(x-y)$, $t > 0$, $x, y \in \R$. Clearly, the transition density of $Z(t)$ is given by $\prod_{j = 1}^d g_t(x_j-y_j)$.

The main result of this paper is the following theorem.

\begin{theorem}
\label{mainthm} (i) The strong Markov process $X(t)$ formed by the unique weak solution to SDE (\ref{main}) has 
a positive jointly continuous transition density function $p^A(t,x,y)$ in $(t,x,y) \in (0,\infty)\times\Rd\times\Rd$ with respect to the Lebesgue measure on $\Rd$.

(ii) The transition density solves
\begin{equation}
\label{parabolic}
\frac{\partial}{\partial t}p^A(t,x,y) = \calL p^A(t,\cdot,y)(x), 
\end{equation}
for all $t \in (0,\infty)$ and $x,y \in \R^d$. 

(iii) For any $T > 0$ there exist $c_1 = c_1(T, d, \alpha, b_1, b_2, b_3, \beta) \ge 1$ such that for any $x,y \in \R^d$, $t \in (0,T]$
\begin{equation}
\label{comparability}
c_1^{-1} \prod_{i = 1}^d g_t(x_i - y_i) \le p^A(t,x,y) \le c_1 \prod_{i = 1}^d g_t(x_i - y_i).
\end{equation}

(iv) For any $T > 0$ and $\gamma \in (0,\alpha \wedge 1)$ there exists $c_2 = c_2(T, \gamma, d, \alpha, b_1, b_2, b_3, \beta) > 0$ such that for any $x, x', y \in \R^d$, $t \in (0,T]$
\begin{equation}
\label{pAHolder}
\left| p^A(t,x,y) - p^A(t,x',y)\right| \le c_2 |x - x'|^{\gamma} t^{-\gamma/\alpha} 
\left(\prod_{i = 1}^d g_t(x_i - y_i) + \prod_{i = 1}^d g_t(x'_i - y_i)\right).
\end{equation}

(v) For any $T > 0$ and $\alpha \in (1,2)$ there exist $c_3 = c_3(T, d, \alpha, b_1, b_2, b_3, \beta) > 0$ such that for any $x,y \in \R^d$, $t \in (0,T]$
\begin{equation}
\label{gradient}
\left| \nabla_x p^A(t,x,y) \right| \le c_3 t^{-1/\alpha} p^A(t,x,y).
\end{equation}

\end{theorem}

Systems of stochastic differential equations driven by cylindrical $\alpha$-stable processes have attracted a lot of attention in recent years see e.g. \cite{BC2006,PZ2011,WZ2015,PSXZ2012,Z2013,SX2014}. In \cite{BC2006} Bass and Chen proved existence and uniqueness of weak solutions of systems of SDEs (\ref{main}) under very mild assumptions on matrices $A(x)$ (i.e. they assumed that $A(x)$ are continuous and bounded in $x$ and nondegenerate for each $x$). Our paper may be treated as the first step in studying fine properties of transition densities of systems of SDEs driven by L{\'e}vy processes with singular L{\'e}vy measures. Fine properties of such transition densities are of great interest but in the case of singular L{\'e}vy measures relatively little is known.
We decided to study the particular case of diagonal matrices $A(x)$ in (\ref{main}) because in that case one can obtain sharp two-sided estimates of these densities. It seems that in the case of general non-diagonal matrices in (\ref{main}) such sharp two-sided estimates are impossible to obtain. Nevertheless, we believe that our results will help to obtain qualitative estimates of transition densities also in the case of general matrices in (\ref{main}). 

The direct inspiration to study transition densities of solutions to (\ref{main}) was a question of Zabczyk concerning gradient estimates of these densities. Another source of inspiration was a recent paper \cite{BSK2017} of Bogdan, Knopova and Sztonyk, where  they constructed heat kernels and obtained upper bounds and H{\"o}lder estimates of them for quite general anisotropic space-inhomogeneous non-local operators. However, the considered jump kernels   cannot be ``too singular''. In particular, the results from \cite{BSK2017} can be applied for systems  (\ref{main}) only when $d = 2$ and $\alpha \in (1,2)$ (see the condition $\alpha +\gamma > d$ in the assumption A1 on page 5 in \cite{BSK2017}). Moreover, even for $d = 2$ and $\alpha \in (1,2)$, the obtained estimates  are far from being optimal. 

In our paper we use a very elegant and efficient method developed by Chen and Zhang in \cite{CZ}. Their approach is based on Levi's freezing coefficient argument (cf. \cite{L1907,F1975,LSU1968}). In \cite{CZ} the non-local and non-symmetric L{\'e}vy type operators on $\Rd$ are studied with jump kernels of the type $\kappa(x,z)/|z|^{d+\alpha}$, $\alpha \in (0,2)$. It turned out that similar ideas  can be applied also in our situation where jump kernels are much more singular.  We follow the road-map from \cite{CZ} however, due to a specific structure of the operator $\calL$, there are many differences between our paper and \cite{CZ}. The main new elements, in comparison to \cite{CZ}, are the proof of crucial Theorem \ref{thmq}, the proof of Lemma \ref{interchange}, the estimates (\ref{B1inequality}-\ref{B4inequality}) and the proof of lower bound estimates of $p^A(t,x,y)$. It is worth  pointing out that in our paper we have shown that the transition density $p^A(t,x,y)$ satisfies the equation (\ref{parabolic}) for all $x,y \in \Rd$ while in  \cite{CZ} it is shown that the heat kernel $p_{\alpha}^{\kappa}(t,x,y)$ satisfies the analogous equation only when $x \ne y$. A similar remark concerns gradient estimates of $p^A(t,x,y)$, which we managed to show for all $x,y \in \Rd$. On the other hand, we were able to prove gradient estimates of $p^A(t,x,y)$ only for $\alpha \in (1,2)$ (in \cite{CZ} gradient estimates were obtained for $\alpha \in [1,2)$). It is worth mentioning  that quite recently a very interesting generalization of the results from \cite{CZ} appeared in \cite{KSV2018}.

The problem of estimates of transition densities for jump L{\'e}vy  and L{\'e}vy-type processes has been intensively studied in recent years see e.g. \cite{CZ, KSV2018, BGR2014, CK2008, CKK2011,KS2012,KS2015,K2014,M2012,KK2011}. However, relatively few results concern processes with jump kernels which are not comparable to isotropic ones. We have already mentioned here the paper \cite{BSK2017}. One should also mention the papers by Sztonyk et al. \cite{BS2007,KS2014,KS2015,S2017} but they only concern heat kernels of translation invariant generators and convolution semigroups for which the existence and many properties follow by Fourier methods. There are also known estimates of anisotropic non-convolution heat kernels given in \cite{S2010a,KS2013} however these are obtained under the assumption that the jump kernel is dominated by that of the rotation invariant stable process. For estimates of derivatives of L{\'e}vy densities we refer the reader to \cite{S2010b,BJ2007,SSW2012,KS2015,KR2016,K2014,CZ}. 

Some estimates of transition densities for processes which are solutions of systems of SDEs driven by L{\'e}vy processes with singular L{\'e}vy measures were obtained in \cite{P1997a,P1997b,P1996,I2001}. However, the results from \cite{P1997a}, when applied to system (\ref{main}), do not imply such sharp estimates which are obtained in Theorem \ref{mainthm}. In particular, they can be applied to system (\ref{main}) only when $x \to a_{ii}(x)$ are $C^{\infty}(\Rd)$ functions. What is more, even in this case, 
the upper bound estimates are of the form $\sup_{x,y \in \Rd} p^A(t,x,y) \le c t^{-d/\alpha}$, while the lower bound estimates of $p^A(t,x,y)$  are also much less precise than ours. They are precise  only for $x=y$, in which case it follows from \cite{P1997a} that $p^A(t,x,x) \approx t^{-d/\alpha}$. The results from \cite{P1997b,P1996,I2001} cannot be applied to system (\ref{main}).

The paper is organized as follows. In Section 2 we introduce the notation and collect known facts needed in the sequel. In Section 3 we construct the function $p^A(t,x,y)$ in terms  of the perturbation series $q(t,x,y)=\sum_{n = 0}^{\infty} q_n(t,x,y)$  using Picard's iteration. In  Theorem \ref{thmq}  we obtain the estimates of $q(t,x,y)$ which are absolutely crucial for the rest of the paper. 
In Section 4 we show that the semigroup defined by $P_t^A f(x) = \int_{\Rd} p^A(t,x,y) f(y) \, dy$ is a Feller semigroup.
 Next, applying  \cite{BC2006}, we argue that  $p^A(t,x,y)$ is, in fact, the transition density of the solution of system (\ref{main}) and we prove most parts of the main theorem. 
In Section 5 we 
 show lower bound estimates of $p^A(t,x,y)$ by using probabilistic arguments. 

\section{Preliminaries}
All constants appearing in this paper are positive and finite. In the whole paper we fix $T > 0$, $d \ge 2$, $d \in \N$, $\alpha \in (0,2)$, $b_1, b_2, b_3, \beta$, where $b_1, b_2, b_3, \beta$ appear in (\ref{bounded}) and (\ref{aHolder}). We adopt the convention that constants denoted by $c$ (or $c_1, c_2, \ldots$) may change their value from one use to the next. In the whole paper, unless is explicitly stated otherwise, we understand that constants denoted by $c$ (or $c_1, c_2, \ldots$) depend on $T, d, \alpha, b_1, b_2, b_3, \beta$. We  also understand that they may depend on the choice of the constant $\gamma \in (0,\beta)$ (or $\gamma \in (0,\alpha \wedge 1)$). We write $f(x) \approx g(x)$ for $x \in A$ if $f, g \ge 0$ on $A$ and there is a constant $c \ge 1$ such that $c^{-1} f(x) \le g(x) \le c f(x)$ for $x \in A$.

Denote
$$
\s_i(x) = a_{ii}^{\alpha}(x).
$$
Note that there exists $c$ such that for any $x, y \in \R^d$ we have
\begin{equation}
\label{sigmaHolder}
|\s_{i}(x) - \s_{i}(y)| \le c \left(|x - y|^{\beta} \wedge 1\right).
\end{equation}
By simple change of variable we get
$$
\calL f(x) = \sum_{i = 1}^d  \lim_{\varepsilon \to 0^+} \frac{\calA_{\alpha}}{2} \int_{|z_i| > \eps} \left[f(x + e_i z_i) + f(x - e_i z_i) - 2 f(x)\right] \, \s_i(x) \frac{dz_i}{|z_i|^{1 + \alpha}}.
$$

Let us introduce some notation which was used in \cite{CZ}. For a function $f:\R^d \to \R$ we denote
$$
\delta_f(x,z) = f(x+z) + f(x-z) - 2f(x).
$$
Similarly, for a function $f:\R_+ \times \R^d \to \R$ we write
$$
\delta_f(t,x,z) = f(t,x+z) + f(t,x-z) - 2f(t,x).
$$
We also denote
$$
\rho_{\gamma}^{\beta}(t,x) = t^{{\gamma}/{\alpha}}(|x|^{\beta} \wedge 1)(t^{1/\alpha} + |x|)^{-1-\alpha}, \quad t > 0,\, x \in \R.
$$
It is well known that
\begin{equation}
\label{standard}
g_t(x) \approx \rho_{\alpha}^{0}(t,x) \quad \quad t > 0, x \in \R.
\end{equation}

One of the most important tools used in our paper are convolution estimates \cite[(2.3-2.4)]{CZ}. They are similar to \cite[Lemma 1.4]{K2000} and \cite[Lemma 2.3]{XZ2014}. In \cite{CZ} they are stated for $t \in (0,1]$. It is easy to check that they hold also for $t \in (0,T]$. For reader's convenience we collected them in Lemma \ref{conv}.

\begin{lemma}
\label{conv}
\begin{itemize}
	\item[(i)] There is $C=C(\alpha)$ such that for any $t>0$ and any\\ $\beta_1\in [0,\alpha/2], \gamma_1\in \R$, 
	\begin{equation}
\int_{\R} \rho_{\gamma_1}^{\beta_1}(t,z)  \, d z \le C t^{\frac{\gamma_1+\beta_1-\alpha}\alpha}.
\label{conv0}
\end{equation}
\item[(ii)] For $T>0$ there is $C=C(\alpha,T)$ such that for any $0<s<t\le T$, $x\in \R$ and any $\beta_1, \beta_2\in [0,\alpha/4],
 \gamma_1,  \gamma_2 \in \R$ we have 
\begin{eqnarray}
&&\int_{\R} \rho_{\gamma_1}^{\beta_1}(t-s,x-z) \rho_{\gamma_2}^{\beta_2}(s,z) \, d z \nonumber\\
&&\le C\left[ (t-s)^{\frac{\gamma_1+\beta_1+\beta_2-\alpha}\alpha}s^{\frac{\gamma_2}\alpha}+(t-s)^{\frac{\gamma_1}\alpha}s^{\frac{\gamma_2+\beta_1+\beta_2-\alpha}\alpha}\right]\rho_{0}^0(t,x)\nonumber\\
&& + C \left[(t-s)^{\frac{\gamma_1+\beta_1-\alpha}\alpha}s^{\frac{\gamma_2}\alpha}\rho_{0}^{\beta_2}(t,x)+(t-s)^{\frac{\gamma_1}\alpha}
s^{\frac{\gamma_2+\beta_2-\alpha}\alpha}\rho_{0}^{\beta_1}(t,x)\right].\label{conv1}
\end{eqnarray}

\item[(iii)] For $T>0$ there is $C=C(\alpha, T)$ such that for any $0<t\le T$, $x\in \R$ and any $\beta_1, \beta_2\in [0,\alpha/4],
 \gamma_1,  \gamma_2 \in \R$ with $\gamma_1+\beta_1>0$ and $\gamma_2+\beta_2>0$  we have 
\begin{eqnarray}
&&\int_{0}^t \int_{\R} \rho_{\gamma_1}^{\beta_1}(t-s,x-z) \rho_{\gamma_2}^{\beta_2}(s,z) \, d z \, d s\nonumber\\
&&\le C \calB\left(\frac{\gamma_1+\beta_1}\alpha, \frac{\gamma_2+\beta_2}\alpha\right)
 \left( \rho_{\gamma_1+\gamma_2+\beta_1+\beta_2}^0+\rho_{\gamma_1+\gamma_2+\beta_2}^{\beta_1}+\rho_{\gamma_1+\gamma_2+\beta_1}^{\beta_2}\right)(t,x)\label{conv2},
\end{eqnarray}
where $\calB(u,w)$ is the Beta function.
\end{itemize}
\end{lemma}

Similarly as in \cite{CZ} we introduce, for $y \in \R^d$, the freezing operator $\calL^{y}$   by
$$
\calL^{y} f(x) = \frac{\calA_{\alpha}}{2} \sum_{i = 1}^d \int_{\R} \delta_f(x,e_i z_i) \, \s_i(y) \frac{dz_i}{|z_i|^{1 + \alpha}}
$$
and
$$
\calL^{y} f(t,x) = \frac{\calA_{\alpha}}{2} \sum_{i = 1}^d \int_{\R} \delta_f(t,x,e_i z_i) \, \s_i(y) \frac{dz_i}{|z_i|^{1 + \alpha}}.
$$
Put
$$
p_y(t,x) = \prod_{i = 1}^d \frac{1}{a_{ii}(y)} g_t\left(\frac{x_i}{a_{ii}(y)}\right).
$$
It is clear that $p_y(t,x)$ is the heat kernel of the operator $\calL^{y}$. In particular, we have
\begin{equation}
\label{pyparabolic}
\frac{\partial}{\partial t} p_y(t,x) = \calL^{y}p_y(t,x), \quad \quad t > 0,\, x, y \in \R^d.
\end{equation}

In the sequel we will use the following standard estimate. For any $\gamma \in (0,1]$ there exists $c = c(\gamma)$ such that for any $\theta \ge 1$ we have 
\begin{equation}
\label{betain}
\int_0^t (t - s)^{\gamma - 1} s^{\theta - 1} \, ds \le \frac{c}{\theta^{\gamma}} t^{(\gamma - 1) + (\theta - 1) + 1}.
\end{equation}

We use the notation $\N_0 = \N \cup \{0\}$.

\section{Upper bound estimates}
The main aim of this section is to construct the function $p^A(t,x,y)$.  This is done by using Levi's method. Is is worth  mentioning that this method was used in the framework of pseudodifferential operators by Kochubei \cite{K1989}. In recent years it was used in several papers to study gradient and Schr{\"o}dinger perturbations of fractional Laplacians and relativistic stable operators see e.g. \cite{BJ2007,JS2012,JS2010,CH,CKS2012,CKS2015,XZ2014}. As we have already mentioned we use the approach by Chen and Zhang \cite{CZ}. It is worth  adding that in \cite{CZ}, in contrast to previous papers, a new way of ``freezing'' coefficient was used. 

Now, we briefly present the main steps used in this section. We define $p^A(t,x,y)$ by (\ref{defpA}). Heuristically, $p^A(t,x,y)$ is equal to the transition density $p_y(t,x-y)$ (of the process with the ``frozen'' jump measure  corresponding to the generator  
$\calL^y$) plus some correction $\int_0^t \int_{\Rd} p_z(t-s,x,z)q(s,z,y) \, dz \, ds$, which is given in terms of the perturbation series $q(t,x,y) = \sum_{n = 0}^{\infty} q_n(t,x,y)$. The most difficult result in this section is Theorem \ref{thmq} which gives upper bound estimates of $q(t,x,y)$. Due to a different structure of the generator $\calL$ in comparison to the L{\'e}vy-type operator $\calL_{\alpha}^{\kappa}$ from \cite{CZ} there are essential differences between our proof and analogous proof in \cite{CZ}, see in particular the definition of the auxiliary function $H_k^L(t,x,y)$ and the induction proof of (\ref{qn}). The next important step in this section is Theorem \ref{Holder} where we derive H{\"o}lder type estimates of $q(t,x,y)$. We also show crucial Lemma \ref{gradientvarphi} which is the main step in obtaining gradient estimates of $p^A(t,x,y)$.

For $x, y \in \R^d$, $t > 0$, let
$$
q_0(t,x,y) = \left(\calL^{x} - \calL^{y}\right) p_y(t,\cdot)(x-y)
$$
and for $n \in \N$ let
\begin{equation}
\label{defqn}
q_n(t,x,y) = \int_0^t \int_{\R^d} q_0(t-s,x,z)q_{n-1}(s,z,y) \, dz \, ds.
\end{equation}
For $x, y \in \R^d$, $t > 0$ we define
\begin{equation}
\nonumber
q(t,x,y) = \sum_{n = 0}^{\infty} q_n(t,x,y)
\end{equation}
and 
\begin{equation}
\label{defpA}
p^A(t,x,y) = p_y(t,x-y) + \int_0^t \int_{\Rd} p_z(t-s,x,z)q(s,z,y) \, dz \, ds.
\end{equation}

By \cite[(2.28)]{CZ} and (\ref{standard}) one easily obtains
\begin{lemma}
\label{generatorpy1}
For any $t \in (0,T]$ and $x, y \in \R^d$ we have
$$
\sum_{k = 1}^d \int_{\R} \left| \delta_{p_y}(t,x,z_k e_k) \right| \frac{d z_k}{|z_k|^{1 + \alpha}}
\le c t^{d - 1} \prod_{i = 1}^d \rho_0^0(t,x_i).
$$
\end{lemma}

An immediate consequence of the above lemma and (\ref{pyparabolic}) is the following estimate 
\begin{equation}
\label{dert}\left|\frac{\partial}{\partial t} p_y(t,x)\right|\le  ct^{d - 1} \prod_{i = 1}^d \rho_0^0(t,x_i),
\end{equation}
for $t \in (0,T]$, $x, y \in \R^d$.

\begin{theorem}
\label{thmq}
The series $\sum_{n = 0}^{\infty} q_n(t,x,y)$ is absolutely and locally uniformly convergent on $(0,T]\times\R^d\times\R^d$. For any $x,y \in \R^d$, $t \in (0,T]$ we have
\begin{equation}
\label{qest}
|q(t,x,y)| \le c t^{d - 1} \left[\prod_{i = 1}^d \rho_0^0(t,x_i-y_i)\right] \left[t^{\beta/\alpha} + \sum_{m = 1}^d \left(|x_m-y_m|^{\beta} \wedge 1\right)\right].
\end{equation}
Moreover, $q(t,x,y)$ is jointly continuous in $(t,x,y) \in (0,T]\times\Rd\times\Rd$.
\end{theorem}
\begin{proof}
By (\ref{sigmaHolder}) and then Lemma \ref{generatorpy1} we get 
\begin{eqnarray}
|q_0(t,x,y)| &\le&
\frac{\calA_{\alpha}}{2} \sum_{k = 1}^d \int_{\R} \left|\delta_{p_y}(t,x-y,e_k z_k)\right| \left|\s_k(x) - \s_k(y) \right| \, \frac{dz_k}{|z_k|^{1 + \alpha}}\nonumber\\
&\le& c \left(|x-y|^{\beta} \wedge 1\right) \sum_{k = 1}^d \int_{\R} \left|\delta_{p_y}(t,x-y,e_k z_k)\right| \, \frac{dz_k}{|z_k|^{1 + \alpha}}\nonumber\\
&\le& M t^{d-1} \left[\sum_{m = 1}^d \left(|x_m-y_m|^{\beta} \wedge 1\right)\right] \left[\prod_{k = 1}^d \rho_0^0(t,x_k-y_k)\right], \label{q0}
\end{eqnarray}
where $M =M(T, d, \alpha, b_1, b_2, b_3, \beta)$.

Put 
$$
\text{I} = \{L=(l_1,\ldots,l_d): \, \forall i \in \{1,\ldots,d\} \,\, l_i = 0 \,\,\, \text{or} \,\,\, l_i = \beta\}.
$$
For any $L=(l_1,\ldots,l_d) \in \text{I}$ denote
$$
|L| = \frac{1}{\beta} \sum_{i = 1}^d l_i.
$$
For $k \in \N_0$ and $L=(l_1,\ldots,l_d) \in \text{I}$ put
$$
H_k^L(t,x,y) = t^{d-1+k\beta/\alpha} \left[\prod_{i = 1}^d \rho_0^0(t,x_i-y_i)\right] 
\left[\prod_{j = 1}^d \left(|x_j-y_j|^{l_j} \wedge 1\right)\right].
$$
We will show that there is $C=C(T, d, \alpha, b_1, b_2, b_3, \beta)$ such that  for any $n \in \N_0$, $x,y \in \R^d$, $t \in (0,T]$,
\begin{equation}
\label{qn}
|q_n(t,x,y)| \le  \frac{MC^n}{\left((n+1)!\right)^{\beta/\alpha}} \sum_{\substack{k \in \N_0, \, L \in \text{I} \\ k + |L| = n+1}} H_k^L(t,x,y),
\end{equation}
where $M$ is the constant from (\ref{q0}).
Let 
%
%
$$
D(t,x,y,m,k,L) = M \int_0^t \int_{\R^d} H_0^{L_m}(t-s,x,z)  H_k^L(s,z,y) \, dz \, ds,
$$
where $L_m\in \text{I}$ is such that $l_m=\beta$ and $|L_m|=1$. Observe that (\ref{q0}) can be rewritten as 

\begin{eqnarray}
|q_0(t,x,y)| &\le&  M \sum^d_{m=1} H_0^{L_m}(t,x, y). \label{q01}
\end{eqnarray}
We will prove (\ref{qn}) by induction. 
The main step consists of proving that for any $n\in \N$,
\begin{equation}
\label{qn11}
\sum_{m = 1}^d \sum_{\substack{k \in \N_0, \, L \in \text{I} \\ k + |L| = n+1}} D(t,x,y,m,k,L)\le \frac{C}{(n+2)^{\beta/\alpha}} \sum_{\substack{k \in \N_0, \, L \in \text{I} \\ k + |L| = n+2}} H_k^L(t,x,y).
\end{equation}

For $n = 0$ the estimate  (\ref{qn}) holds by (\ref{q01}). 


 Assume that (\ref{qn}) holds for some $n \in \N_0$.  By (\ref{defqn}), (\ref{q01}) and our induction hypothesis we obtain
\begin{equation}
\label{qn1}
|q_{n+1}(t,x,y)| \le  \frac{MC^n}{\left((n+1)!\right)^{\beta/\alpha}} \sum_{m = 1}^d \sum_{\substack{k \in \N_0, \, L \in \text{I} \\ k + |L| = n+1}} D(t,x,y,m,k,L).
\end{equation}
Then, if (\ref{qn11}) is true,  then  (\ref{qn}) holds for $n+1$. Hence in order to complete the proof it is enough to show  (\ref{qn11}).

To this end we consider 3 cases.

\vskip 10pt

{\bf{Case 1.}} $L = (0,\ldots,0)$, $k = n+1$.

We have
\begin{eqnarray*}
D(t,x,y,m,k,L) &=& M \int_0^t s^{n \beta/\alpha} \left[\prod_{\substack{i = 1\\ i \ne m}}^d \int_{\R} 
\rho_{\alpha}^0(t-s,x_i-z_i) \rho_{\alpha}^0(s,z_i-y_i) \, d z_i\right] \\
&& \times \int_{\R} \rho_{0}^{\beta}(t-s,x_m-z_m) \rho_{\beta}^0(s,z_m-y_m) \, d z_m \, ds.
\end{eqnarray*}

By (\ref{conv1}),  we obtain
$$
\int_{\R} \rho_{\alpha}^0(t-s,x_i-z_i) \rho_{\alpha}^0(s,z_i-y_i) \, d z_i \le c \rho_{\alpha}^0(t,x_i-y_i)
$$
and
\begin{eqnarray*}
&& \int_{\R} \rho_{0}^{\beta}(t-s,x_m-z_m) \rho_{\beta}^0(s,z_m-y_m) \, d z_m\\
&& \le c \left[(t-s)^{(\beta - \alpha)/\alpha} s^{\beta/\alpha} \rho_{0}^0(t,x_m-y_m)
+ s^{(2\beta - \alpha)/\alpha} \rho_{0}^0(t,x_m-y_m) \right.\\
&&
\left. \,\,\,\,\,\, + s^{(\beta - \alpha)/\alpha} \rho_{0}^{\beta}(t,x_m-y_m)\right].
\end{eqnarray*}
Hence
\begin{eqnarray*}
D(t,x,y,m,k,L) &\le& c t^{d-1} \left[\prod_{i = 1}^d \rho_0^0(t,x_i-y_i)\right]\\
&& \times \left[\int_0^t (t-s)^{(\beta - \alpha)/\alpha} s^{(n+1)\beta/\alpha} \,ds
+ \int_0^t s^{((n+2)\beta - \alpha)/\alpha} \,ds \right. \\
&&
\left. \,\,\,\, + \int_0^t s^{((n+1)\beta - \alpha)/\alpha} \,ds \left(|x_m-y_m|^{\beta} \wedge 1\right)\right].
\end{eqnarray*}
By (\ref{betain}) this implies that 
\begin{eqnarray}
\nonumber
D(t,x,y,m,k,L) &\le& \frac{c}{(n+1)^{\beta/\alpha}} t^{d-1} \left[\prod_{i = 1}^d \rho_0^0(t,x_i-y_i)\right] \\
\label{Case1in}
&&  \times \left[t^{(n+2)\beta/\alpha} +  t^{(n+1)\beta/\alpha}\left(|x_m-y_m|^{\beta} \wedge 1\right)\right].
\end{eqnarray}

\vskip 10pt

{\bf{Case 2.}} $L = (l_1,\ldots,l_d) \ne (0,\ldots,0)$, $l_m = 0$.

Put $Z(L) = \{i \in \{1,\ldots,d\}: \, l_i = \beta\}$ and $i(L) = \inf Z(L)$. Clearly $m \notin Z(L)$. We have
\begin{eqnarray*}
&& D(t,x,y,m,k,L)
= M \int_0^t \int_{\R} \rho_{0}^{\beta}(t-s,x_m-z_m) \rho_{\alpha}^0(s,z_m-y_m) \, d z_m\\
&& \times \int_{\R} \rho_{\alpha}^0(t-s,x_{i(L)}-z_{i(L)}) \rho_{0}^{\beta}(s,z_{i(L)}-y_{i(L)}) \, d z_{i(L)}\\
&& \times \left[ \prod_{\substack{i \in Z(L)\\ i \ne i(L)}} 
\int_{\R} \rho_{\alpha}^0(t-s,x_{i}-z_{i}) \rho_{\alpha}^{\beta}(s,z_{i}-y_{i}) \, d z_{i} \right] \\
&& \times \left[ \prod_{\substack{i \notin Z(L)\\ i \ne m}} 
\int_{\R} \rho_{\alpha}^0(t-s,x_{i}-z_{i}) \rho_{\alpha}^{0}(s,z_{i}-y_{i}) \, d z_{i} \right] s^{k\beta/\alpha} \, ds.
\end{eqnarray*}
By  (\ref{conv1}) this is bounded from above by 
\begin{eqnarray*}
&& c\int_0^t \left[(t-s)^{(\beta-\alpha)/\alpha} s \rho_{0}^{0}(t,x_m-y_m) 
+ t^{\beta/\alpha}  \rho_{0}^{0}(t,x_m-y_m) 
+ \rho_{0}^{\beta}(t,x_m-y_m) \right]\\
&& \times \left[t^{\beta/\alpha} \rho_{0}^{0}(t,x_{i(L)}-y_{i(L)})
+ t  s^{(\beta-\alpha)/\alpha} \rho_{0}^{0}(t,x_{i(L)}-y_{i(L)}) 
+ \rho_{0}^{\beta}(t,x_{i(L)}-y_{i(L)})\right]\\
&& \times \left[\prod_{\substack{i \in Z(L)\\ i \ne i(L)}} 
\left[t^{(\alpha + \beta)/\alpha} \rho_{0}^{0}(t,x_i-y_i) + t \rho_{0}^{\beta}(t,x_i-y_i)\right]\right]\\
&& \times \left[\prod_{\substack{i \notin Z(L)\\ i \ne m}}
t \rho_{0}^{0}(t,x_i-y_i) \right] s^{k\beta/\alpha} \, ds.
\end{eqnarray*}
Note that $\# Z(L) = |L|$. We have
\begin{eqnarray}
\nonumber
&& \prod_{\substack{i \in Z(L)\\ i \ne i(L)}} 
\left[t^{(\alpha + \beta)/\alpha} \rho_{0}^{0}(t,x_i-y_i) + t \rho_{0}^{\beta}(t,x_i-y_i)\right] \le c t^{|L| - 1} 
\left(\prod_{\substack{i \in Z(L)\\ i \ne i(L)}} \rho_{0}^{0}(t,x_i-y_i)\right)\\
\label{ZLin}
&& \times
\sum_{\substack{r \le |L| - 1\\ r \in \N_0}} \,\, \sum_{\{k_1,\ldots,k_r\} \subset Z(L) \setminus \{i(L)\}} 
t^{(|L| - r - 1)\beta/\alpha}
\prod_{i = 1}^r \left(|x_{k_i} - y_{k_i}|^{\beta} \wedge 1\right),
\end{eqnarray}
where for $r = 0$ we understand that
$
\prod_{i = 1}^r \left(|x_{k_i} - y_{k_i}|^{\beta} \wedge 1\right) = 1.
$
It follows that
\begin{eqnarray*}
&& D(t,x,y,m,k,L) \le c t^{d-2} \left[\prod_{i = 1}^d \rho_0^0(t,x_i-y_i)\right]\\
&& \times
\sum_{\substack{r \le |L| - 1\\ r \in \N_0}} \,\, \sum_{\{k_1,\ldots,k_r\} \subset Z(L) \setminus \{i(L)\}} 
t^{(|L| - r - 1)\beta/\alpha}
\left[\prod_{i = 1}^r \left(|x_{k_i} - y_{k_i}|^{\beta} \wedge 1\right)\right] \\
&& \times \left[t^{(\alpha + \beta)/\alpha} \int_0^t (t-s)^{(\beta - \alpha)/\alpha} s^{k \beta/\alpha} \, ds
+ t \int_0^t (t-s)^{(\beta - \alpha)/\alpha} s^{(\beta + k \beta)/\alpha} \, ds \right. \\
&& + t \int_0^t (t-s)^{(\beta - \alpha)/\alpha} s^{k \beta/\alpha} \, ds \left(|x_{i(L)} - y_{i(L)}|^{\beta} \wedge 1\right)  \\
&& + t^{2\beta/\alpha} \int_0^t  s^{k \beta/\alpha} \, ds
+ t^{(\alpha + \beta)/\alpha} \int_0^t s^{(\beta - \alpha + k \beta)/\alpha} \, ds
+ t^{\beta/\alpha} \int_0^t  s^{k \beta/\alpha} \, ds \left(|x_{i(L)} - y_{i(L)}|^{\beta} \wedge 1\right)\\
&& \times \left[t^{\beta/\alpha} \int_0^t  s^{k \beta/\alpha} \, ds
+ t \int_0^t s^{(\beta - \alpha + k \beta)/\alpha} \, ds
+\int_0^t  s^{k \beta/\alpha} \, ds \left(|x_{i(L)} - y_{i(L)}|^{\beta} \wedge 1\right)\right]\\ 
&& \times \left. \left(|x_{m} - y_{m}|^{\beta} \wedge 1\right)\right].
\end{eqnarray*}
Using this and (\ref{betain}) we get
\begin{eqnarray*}
&& D(t,x,y,m,k,L) \le c t^{d-2} \left[\prod_{i = 1}^d \rho_0^0(t,x_i-y_i)\right]\\
&& \times
\sum_{\substack{r \le |L| - 1\\ r \in \N_0}} \,\, \sum_{\{k_1,\ldots,k_r\} \subset Z(L) \setminus \{i(L)\}} 
t^{(|L| - r - 1)\beta/\alpha}
\left[\prod_{i = 1}^r \left(|x_{k_i} - y_{k_i}|^{\beta} \wedge 1\right)\right]\\
&& \times \frac{1}{(k+1)^{\beta/\alpha}} 
\left[ t^{(\alpha + k \beta + 2 \beta)/\alpha} 
+ t^{(\alpha + k \beta + \beta)/\alpha} \left(|x_{i(L)} - y_{i(L)}|^{\beta} \wedge 1\right) \right. \\
&& \left. t^{(\alpha + k \beta + \beta)/\alpha} \left(|x_{m} - y_{m}|^{\beta} \wedge 1\right)
+ t^{(\alpha + k \beta)/\alpha} \left(|x_{i(L)} - y_{i(L)}|^{\beta} \wedge 1\right) \left(|x_{m} - y_{m}|^{\beta} \wedge 1\right)\right].
\end{eqnarray*}
Note that $k + |L| = n+1$. It follows that 
\begin{eqnarray}
\nonumber
&& D(t,x,y,m,k,L) \le \frac{c t^{d-1}}{(k+1)^{\beta/\alpha}} \left[\prod_{i = 1}^d \rho_0^0(t,x_i-y_i)\right]\\
\label{Case2in}
&& \times \sum_{\substack{r \le |L| + 1\\ r \in \N_0}} \,\, \sum_{\{k_1,\ldots,k_r\} \subset Z(L) \cup \{m\}} 
t^{(n+2-r)\beta/\alpha} \prod_{i = 1}^r \left(|x_{k_i} - y_{k_i}|^{\beta} \wedge 1\right).
\end{eqnarray}

\vskip 10pt

{\bf{Case 3.}} $L = (l_1,\ldots,l_d)$, $l_m = \beta$.

We have
\begin{eqnarray*}
&& D(t,x,y,m,k,L)
= M \int_0^t \int_{\R} \rho_{0}^{\beta}(t-s,x_m-z_m) \rho_{0}^{\beta}(s,z_m-y_m) \, d z_m\\
&& \times \left[\prod_{\substack{i \in Z(L)\\ i \ne m}} 
\int_{\R} \rho_{\alpha}^0(t-s,x_{i}-z_{i}) \rho_{\alpha}^{\beta}(s,z_{i}-y_{i}) \, d z_{i}\right] \\
&& \times \left[\prod_{i \notin Z(L)}
\int_{\R} \rho_{\alpha}^0(t-s,x_{i}-z_{i}) \rho_{\alpha}^{0}(s,z_{i}-y_{i}) \, d z_{i}\right] s^{k\beta/\alpha} \, ds.
\end{eqnarray*}
By  $(\ref{conv1})$, this is bounded from above by 
\begin{eqnarray*}
&& c\int_0^t \left[(t-s)^{(2\beta-\alpha)/\alpha} \rho_{0}^{0}(t,x_m-y_m) 
+ s^{(2\beta-\alpha)/\alpha} \rho_{0}^{0}(t,x_m-y_m) \right. \\
&& \left. + (t-s)^{(\beta-\alpha)/\alpha} \rho_{0}^{\beta}(t,x_m-y_m)
+ s^{(\beta-\alpha)/\alpha} \rho_{0}^{\beta}(t,x_m-y_m) \right] \\
&&\times \left[\prod_{\substack{i \in Z(L)\\ i \ne m}} 
\left[t^{(\alpha + \beta)/\alpha} \rho_{0}^{0}(t,x_i-y_i) + t \rho_{0}^{\beta}(t,x_i-y_i)\right]\right]\\
&&\times \left[\prod_{i \notin Z(L)}
t \rho_{0}^{0}(t,x_i-y_i) \right] s^{k\beta/\alpha} \, ds.
\end{eqnarray*}
Using similar reasoning as in (\ref{ZLin}) this is bounded from above by 
\begin{eqnarray*}
&& c t^{d-1} \left[\prod_{i = 1}^d \rho_0^0(t,x_i-y_i)\right]\\
&& \times
\sum_{\substack{r \le |L| - 1\\ r \in \N_0}} \,\, \sum_{\{k_1,\ldots,k_r\} \subset Z(L) \setminus \{m\}} 
t^{(|L| - r - 1)\beta/\alpha}
\prod_{i = 1}^r \left(|x_{k_i} - y_{k_i}|^{\beta} \wedge 1\right)\\
&&\left[\int_0^t (t-s)^{(2\beta - \alpha)/\alpha} s^{k \beta/\alpha} \, ds
+ \int_0^t  s^{(2\beta - \alpha + k \beta)/\alpha} \, ds \right. \\
&& + \int_0^t (t-s)^{(\beta - \alpha)/\alpha} s^{k \beta/\alpha} \, ds \left(|x_{m} - y_{m}|^{\beta} \wedge 1\right)  \\
&& \left. \int_0^t  s^{(\beta - \alpha + k \beta)/\alpha} \, ds \left(|x_{m} - y_{m}|^{\beta} \wedge 1\right) \right] 
\end{eqnarray*}
By (\ref{betain}) it follows that 
\begin{eqnarray*}
&& D(t,x,y,m,k,L) \le c t^{d-1} \left[\prod_{i = 1}^d \rho_0^0(t,x_i-y_i)\right]\\
&& \times
\sum_{\substack{r \le |L| - 1\\ r \in \N_0}} \,\, \sum_{\{k_1,\ldots,k_r\} \subset Z(L) \setminus \{m\}} 
t^{(|L| - r - 1)\beta/\alpha}
\prod_{i = 1}^r \left(|x_{k_i} - y_{k_i}|^{\beta} \wedge 1\right)\\
&& \frac{1}{(k+1)^{\beta/\alpha}} 
\left[ t^{(k \beta + 2 \beta)/\alpha} 
+ t^{(k \beta + \beta)/\alpha} \left(|x_{m} - y_{m}|^{\beta} \wedge 1\right) \right] 
\end{eqnarray*}
Hence
\begin{eqnarray}
\nonumber
&& D(t,x,y,m,k,L) \le \frac{c t^{d-1}}{(k+1)^{\beta/\alpha}} \left[\prod_{i = 1}^d \rho_0^0(t,x_i-y_i)\right]\\
\label{Case3in}
&& \times \sum_{\substack{r \le |L|\\ r \in \N_0}} \sum_{\{k_1,\ldots,k_r\} \subset Z(L)} 
t^{(n+2-r)\beta} \prod_{i = 1}^r \left(|x_{k_i} - y_{k_i}|^{\beta} \wedge 1\right).
\end{eqnarray}

Recall that $n+1 = k+|L|$, $|L| \le d$, so $k \ge n+1-d$. Hence $\frac{1}{(k+1)^{\beta/\alpha}} \le \frac{c}{(n+2)^{\beta/\alpha}}$, where $c = c(d)$. Consequently, (\ref{qn1}), (\ref{Case1in}), (\ref{Case2in}), (\ref{Case3in}) gives that (\ref{qn11}) holds,  which finishes the induction proof. 

From (\ref{qn}) we immediately obtain that for any $n \in \N_0$
\begin{equation}
\label{newqn}
|q_n(t,x,y)| \le
\frac{c C^n}{((n+1)!)^{\beta/\alpha}} t^{d-1} \left[\prod_{i = 1}^d \rho_0^0(t,x_i-y_i)\right]
\left[t^{\beta/\alpha} + \sum_{m = 1}^d \left(|x_m-y_m|^{\beta} \wedge 1\right)\right].
\end{equation}
It follows that $\sum_{n = 0}^{\infty} q_n(t,x,y)$ is absolutely and locally uniformly convergent on $(0,T]\times \R^d \times \R^d$ and (\ref{qest}) holds.

By the properties of $p_y(t,x)$ it is easy to justify that $q_0(t,x,y)$ is jointly continuous in $(t,x,y) \in (0,T] \times \Rd \times \Rd$. By (\ref{defqn}) and induction method the same property holds for $q_n(t,x,y)$ for each $n \in N$. Since  $\sum_{n = 0}^{\infty} q_n(t,x,y)$ is absolutely and locally uniformly convergent we finally obtain that $q(t,x,y)$ is jointly continuous in $(t,x,y) \in (0,T] \times \Rd \times \Rd$.
\end{proof}


By elementary calculations, for any  $t >0$, $u, w \in \R$  satisfying $|u- w| \le t^{1/\alpha}$, we have
\begin{equation}
\label{xxp}
\rho_0^0(t,u) \approx \rho_0^0(t,w).
\end{equation}

\begin{lemma}
\label{productgt1}
There exists $c = c(\alpha,d)$ such that for any $m \in \{1,\ldots,d\}$, $t >0$, $x, x'\in \R^d$   we have
\begin{equation}
\label{productgt2M1}
\left| \prod_{i=1}^m  g_t\left({x_i}\right) - 
\prod_{i=1}^m  g_t\left({x'_i}\right)\right|
\le c  \left(\prod_{i=1}^m g_t(|x_i|\wedge |x'_i|) \right)\left[1\wedge \sum_{j = 1}^m  \frac{|x_j - x'_j|}{t^{1/\alpha}+|x_i|\wedge |x'_i|}\right].
\end{equation}
If additionally $|x - x'| \le t^{1/\alpha}$, then

\begin{equation}
\label{productgt2M2}
\left| \prod_{i=1}^m  g_t\left({x_i}\right) - 
\prod_{i=1}^m  g_t\left({x'_i}\right)\right|
\le c   t^{m} \left(\prod_{i=1}^m \rho_0^0(t,x_i) \right)\left[1\wedge \sum_{j = 1}^m t^{-1/\alpha} |x_j - x'_j|\right].
\end{equation}

\end{lemma}
\begin{proof}
Let $g^{(3)}_t\left(\cdot \right) $ be the radial profile of the transition density of the standard $3$-dimensional $\alpha$-stable isotropic process.
Then it is well known (see e.g. \cite[(11)]{BJ2007}) that 
$$\frac {d g_t(x)}{dx}= - 4\pi x g^{(3)}_t(|x|), \quad x\in \R.$$
By the standard estimates of transition density of the  $\alpha$-stable isotropic process we have  
$$ g^{(3)}_t\left(|x|\right)\le c \frac{g_t(x)}{\left(|x| + t^{1/\alpha}\right)^2},$$
which yields
$$\left|\frac {d g_t(x)}{dx}\right|\le c \frac {|x|}{\left(|x| + t^{1/\alpha}\right)^2}g_t\left(x\right)\le  c \frac {g_t(x)}{|x| + t^{1/\alpha}}, \quad x\in \R.$$

Next, for any $u, w \in \R$, from the above gradient estimate of $g_t$ and the fact that $g_t(u)$ is decreasing in $|u|$,
\begin{equation}
\label{gtdiff1}
\left|g_t\left(u\right) - g_t\left(w\right)\right|
\le c \frac{|u -w|}{|u|\wedge |w| + t^{1/\alpha}} 
g_t\left(|u|\wedge |w|\right).
\end{equation}

Hence, by monotonicity of $g_t(u)$ and (\ref{gtdiff1}),
\begin{eqnarray*}
&&
\left| \prod_{i=1}^{m}  g_t\left({x_i}\right) - 
\prod_{i=1}^{m}  g_t\left({x'_i}\right)\right|\\
&&
\le  \sum_{i=1}^m \left[\left|g_t\left(x_{i}\right) - g_t\left(x'_{i}\right)\right| 
\prod_{j\ne i, 1\le j\le m} g_t\left(|x_{i}|\wedge |x'_{i}|\right) \right]\\
&& \le c \left(\prod_{i=1}^m g_t(|x_i|\wedge |x'_i|) \right) \sum_{j = 1}^m  
\frac{|x_i-x'_i|}{|x_i| \wedge|x'_i| + t^{1/\alpha}}
\end{eqnarray*}

Combining this with the obvious inequality 
$$\left| \prod_{i=1}^{m}  g_t\left({x_i}\right) - 
\prod_{i=1}^{m}  g_t\left({x'_i}\right)\right|\le \prod_{i=1}^m g_t(|x_i|\wedge |x'_i|)$$
we finish the proof  of (\ref{productgt2M1}).

To get the second inequality we apply (\ref{xxp}). \end{proof}

As a direct conclusion of Lemma \ref{productgt1} we get
\begin{corollary}
\label{productgt}
For any $k \in \{1,\ldots,d\}$, $t >0$, $x, x', y \in \R^d$  satisfying $|x - x'| \le t^{1/\alpha}$ we have
\begin{eqnarray*}
&& \left| \prod_{\substack{i=1\\ i \ne k}}^d a_{ii}(y) g_t\left(\frac{x_i}{a_{ii}(y)}\right) - 
\prod_{\substack{i=1\\ i \ne k}}^d a_{ii}(y) g_t\left(\frac{x'_i}{a_{ii}(y)}\right)\right|\\
&& \le c t^{d - 1} \left(\prod_{\substack{i=1\\ i \ne k}}^d \rho_0^0(t,x_i) \right) \sum_{j = 1}^d (t^{-1/\alpha} |x_j - x'_j| \wedge 1).
\end{eqnarray*}
\end{corollary}

\begin{corollary}
\label{productgt2}
For any  $t >0$ and $x, y, w \in \R^d$   we have
$$|p_x(t,w)-p_y(t,w)|\le c p_x(t,w) (|x-y|^\beta \wedge 1).$$
\end{corollary}

\begin{proof} We have
\begin{eqnarray*}
|p_x(t,w)-p_y(t,w)|&=&\left| \prod_{i=1}^d a^{-1}_{ii}(y) g_t\left(\frac{w_i}{a_{ii}(y)}\right) -  \prod_{i=1}^d a^{-1}_{ii}(x) g_t\left(\frac{w_i}{a_{ii}(x)}\right)\right| \\&\le&
\left|\prod_{i=1}^d a^{-1}_{ii}(y) -\prod_{i=1}^d a^{-1}_{ii}(x)\right| \prod_{i=1}^d g_t\left(\frac{w_i}{a_{ii}(y)}\right)\\
&+&\prod_{i=1}^d a^{-1}_{ii}(x) \left|\prod_{i=1}^d g_t\left(\frac{w_i}{a_{ii}(y)}\right) -  \prod_{i=1}^d g_t\left(\frac{w_i}{a_{ii}(x)}\right)\right|
\end{eqnarray*}
Next, by (\ref{aHolder}) and (\ref{bounded}),
$$\left|\prod_{i=1}^d a^{-1}_{ii}(y) -\prod_{i=1}^d a^{-1}_{ii}(x)\right|\le c(|x-y|^\beta \wedge 1), $$
and by Lemma \ref{productgt1} together with (\ref{aHolder}) and (\ref{bounded})

\begin{eqnarray*}&& \left|\prod_{i=1}^d g_t\left(\frac{w_i}{a_{ii}(y)}\right) -  \prod_{i=1}^d g_t\left(\frac{w_i}{a_{ii}(x)}\right)\right|\\
&\le &
c \left[1\wedge \sum_{j = 1}^m  \frac{|w_i||a_{ii}(x)-a_{ii}(y)|}{t^{1/\alpha}+|w_i|}\right]\prod_{i=1}^d g_t\left(\frac{w_i}{a_{ii}(x)}\right)\\
&\le & c (|x-y|^\beta \wedge 1) \prod_{i=1}^d g_t\left(\frac{w_i}{a_{ii}(x)}\right).\end{eqnarray*}
The proof is completed. 
\end{proof}

\begin{corollary}
\label{mintegrallemma}
For any $x \in \Rd$, $t > 0$ we have
\begin{equation}
\label{mintegral}
\int_{\R^{d-1}} \left|\prod_{i=2}^d \frac{1}{a_{ii}(y)}g_t\left(\frac{x_i-y_i}{a_{ii}(y)}\right)
-\prod_{i=2}^d \frac{1}{a_{ii}(\tilde{y})}g_t\left(\frac{x_i-y_i}{a_{ii}(\tilde{y})}\right)\right|
\, dy_2 \, \ldots \, dy_d \le c (|x_1-y_1|^{\beta} \wedge 1),
\end{equation}
where $\tilde{y} = (x_1,y_2,\ldots,y_d)$.
\end{corollary}
\begin{proof}
By the same arguments as in the proof Corollary \ref{productgt2} we have

$$\left|\prod_{i=2}^d \frac{1}{a_{ii}(y)}g_t\left(\frac{x_i-y_i}{a_{ii}(y)}\right)
-\prod_{i=2}^d \frac{1}{a_{ii}(\tilde{y})}g_t\left(\frac{x_i-y_i}{a_{ii}(\tilde{y})}\right)\right|
\le c \left[\prod_{i=2}^d g_t(x_i-y_i)\right](|y-\tilde{y}|^\beta \wedge 1).$$
Observing that $|y-\tilde{y}|= |x_1-y_1| $ we obtain the conclusion by integration. 
\end{proof}

\begin{lemma}
\label{generatorpy2}
For any $t \in (0,T]$, $x, x', y \in \R^d$  satisfying $|x - x'| \le t^{1/\alpha}$ we have
\begin{eqnarray}
\nonumber
&& \sum_{k = 1}^d \int_{\R} \left|\delta_{p_y}(t,x,z_k e_k) - \delta_{p_y}(t,x',z_k e_k)\right| \frac{d z_k}{|z_k|^{1 + \alpha}}\\
\label{gen2}
&& \le c t^{d - 1} \left(\prod_{i = 1}^d \rho_0^0(t,x_i) \right) \sum_{j = 1}^d (t^{-1/\alpha} |x_j - x'_j| \wedge 1).
\end{eqnarray}
\end{lemma}
\begin{proof}
Fix $k \in \{1,\ldots,d\}$
For $t > 0$, $z \in \R$ put
$$
h_y(t,z) = a_{kk}(y) g_t\left(\frac{z}{a_{kk}(y)}\right).
$$

We have
\begin{eqnarray*}
&&\int_{\R} \left|\delta_{p_y}(t,x,z_k e_k) - \delta_{p_y}(t,x',z_k e_k)\right| \frac{d z_k}{|z_k|^{1 + \alpha}}\\
&& = \int_{\R} \left| \left[\prod_{\substack{i=1\\ i \ne k}}^d a_{ii}(y) g_t\left(\frac{x_i}{a_{ii}(y)}\right) \right] \delta_{h_y}(t,x_k,z_k) \right. \\
&& \left.
- \left[\prod_{\substack{i=1\\ i \ne k}}^d a_{ii}(y) g_t\left(\frac{x'_i}{a_{ii}(y)}\right) \right] \delta_{h_y}(t,x'_k,z_k)\right|
 \frac{d z_k}{|z_k|^{1 + \alpha}}\\
&& \le 
\left| \prod_{\substack{i=1\\ i \ne k}}^d a_{ii}(y) g_t\left(\frac{x_i}{a_{ii}(y)}\right) - 
\prod_{\substack{i=1\\ i \ne k}}^d a_{ii}(y) g_t\left(\frac{x'_i}{a_{ii}(y)}\right)\right|
\int_{\R} \left| \delta_{h_y}(t,x_k,z_k) \right| \frac{d z_k}{|z_k|^{1 + \alpha}}\\
&& + \left[ \prod_{\substack{i=1\\ i \ne k}}^d a_{ii}(y) g_t\left(\frac{x'_i}{a_{ii}(y)}\right) \right]
\int_{\R} \left| \delta_{h_y}(t,x_k,z_k) - \delta_{h_y}(t,x'_k,z_k)\right| \frac{d z_k}{|z_k|^{1 + \alpha}}.
\end{eqnarray*}

By \cite[(2.28)]{CZ} we have
$$
\int_{\R} \left| \delta_{h_y}(t,x_k,z_k) \right| \frac{d z_k}{|z_k|^{1 + \alpha}} \le c \rho_0^0(t,x_k),
$$
while, by \cite[(2.29)]{CZ},

$$\int_{\R} \left| \delta_{h_y}(t,x_k,z_k) - \delta_{h_y}(t,x'_k,z_k)\right| \frac{d z_k}{|z_k|^{1 + \alpha}}\le c \rho_0^0(t,x_k)(t^{-1/\alpha} |x_k - x'_k| \wedge 1).$$
Applying the above inequalities  and Corollary \ref{productgt} we obtain the desired bound (\ref{gen2}).
\end{proof}

\begin{lemma}
\label{q0Holder}
For any $x, x', y \in \R^d$, $t \in (0,T]$ and $\gamma \in (0,\beta)$ we have
\begin{eqnarray*}
&& |q_0(t,x,y) - q_0(t,x',y)| \le c \left( |x - x'|^{\beta - \gamma} \wedge 1\right)\\
&& \times \left[ \sum_{k = 1}^d \left[\prod_{\substack{i=1\\ i \ne k}}^d (t \rho_0^0(t,x_i-y_i))\right] 
\left(\rho_{\gamma}^0(t,x_k-y_k) + \rho_{\gamma - \beta}^{\beta}(t,x_k-y_k)\right) \right. \\
&& + \left. \sum_{k = 1}^d \left[\prod_{\substack{i=1\\ i \ne k}}^d (t \rho_0^0(t,x'_i-y_i))\right] 
\left(\rho_{\gamma}^0(t,x'_k-y_k) + \rho_{\gamma - \beta}^{\beta}(t,x'_k-y_k)\right) \right]
\end{eqnarray*}
\end{lemma}
\begin{proof}
\textbf{Case 1.} $|x-x'| > t^{1/\alpha}$.

By (\ref{q0}) we get, for $z=x\ \text{or} \ x'$,
$$
|q_0(t,z,y)| 
\le c \left(|x-x'|^{\beta - \gamma} \wedge 1\right) t^{d-1} \left(\prod_{i=1}^d \rho_0^0(t,z_i-y_i)\right) \left(\sum_{m=1}^d \left(|z_m - y_m| \wedge 1\right)^{\beta} t^{\frac{\gamma - \beta}{\alpha}}\right).
$$

\textbf{Case 2.} $|x - x'| \le t^{1/\alpha}$.

We have
\begin{eqnarray*}
&& |q_0(t,x,y) - q_0(t,x',y)| \\
&& = \frac{\calA}{2} \left| \sum_{k = 1}^d \int_{\R} \delta_{p_{y}}(t,x-y,z_k e_k) (a_{kk}^{\alpha}(x) - a_{kk}^{\alpha}(y)) \frac{d z_k}{|z_k|^{1 + \alpha}}\right. \\
&& - \left. \sum_{k = 1}^d \int_{\R} \delta_{p_{y}}(t,x'-y,z_k e_k) (a_{kk}^{\alpha}(x') - a_{kk}^{\alpha}(y)) \frac{d z_k}{|z_k|^{1 + \alpha}}\right| \\
&& \le \sum_{k = 1}^d \int_{\R} \left| \delta_{p_{y}}(t,x-y,z_k e_k) - \delta_{p_{y}}(t,x'-y,z_k e_k) \right| 
\left|a_{kk}^{\alpha}(x) - a_{kk}^{\alpha}(y)\right| \frac{d z_k}{|z_k|^{1 + \alpha}} \\
&& + \sum_{k = 1}^d \int_{\R} \left| \delta_{p_{y}}(t,x'-y,z_k e_k) \right| 
\left|a_{kk}^{\alpha}(x) - a_{kk}^{\alpha}(x')\right| \frac{d z_k}{|z_k|^{1 + \alpha}} \\
&& \le c \left(|x-y|^{\beta} \wedge 1\right) \sum_{k = 1}^d \int_{\R} 
\left| \delta_{p_{y}}(t,x-y,z_k e_k) - \delta_{p_{y}}(t,x'-y,z_k e_k) \right|  \frac{d z_k}{|z_k|^{1 + \alpha}} \\
&& + c \left(|x-x'|^{\beta} \wedge 1\right) \sum_{k = 1}^d \int_{\R} \left| \delta_{p_{y}}(t,x'-y,z_k e_k) \right| 
 \frac{d z_k}{|z_k|^{1 + \alpha}}\\
&& = \text{I} + \text{II}.
\end{eqnarray*}

By Lemma \ref{generatorpy2} we get
\begin{equation}
\label{Ifor}
\text{I} \le c \left(|x-y|^{\beta} \wedge 1\right) \left(t^{-1/\alpha} |x-x'| \wedge 1\right) t^{d-1} \prod_{i=1}^d \rho_0^0(t,x_i-y_i).
\end{equation}
Note that $t^{-1/\alpha} |x-x'| \le 1$ so $t^{-1/\alpha} |x-x'| \le t^{\frac{\gamma - \beta}{\alpha}} |x - x'|^{\beta - \gamma}$. Let $m \in \{1,\ldots,d\}$ be such that $|x_m - y_m| = \max_{i \in \{1,\ldots,d\}} |x_i - y_i|$. It follows that (\ref{Ifor}) is bounded from above by 

\begin{eqnarray*}
&& c t^{\frac{\gamma - \beta}{\alpha}} |x-x'|^{\beta - \gamma} t^{d-1} \left(\prod_{i=1}^d \rho_0^0(t,x_i-y_i)\right) 
\left(|x_m - y_m|^{\beta} \wedge 1\right)\\
&& \le c \left(|x-x'|^{\beta - \gamma} \wedge 1\right)  \left(\prod_{\substack{i=1\\ i \ne m}}^d t\rho_0^0(t,x_i-y_i)\right) \rho_{\gamma - \beta}^{\beta}(t,x_m-y_m).
\end{eqnarray*}

We have $|x-x'| \le t^{1/\alpha}$ so $1 \le |x-x'|^{-\gamma} t^{\gamma/\alpha}$. It follows that $\left(|x-x'|^{\beta} \wedge 1\right) \le \left(|x-x'|^{\beta - \gamma} \wedge 1\right) t^{\gamma/\alpha}$. Using this and Lemma \ref{generatorpy1} we get
\begin{eqnarray*}
\text{II} &\le& c \left(|x-x'|^{\beta} \wedge 1\right) t^{d-1} \prod_{i=1}^d \rho_0^0(t,x'_i-y_i)\\
&\le& c \left(|x-x'|^{\beta - \gamma} \wedge 1\right) t^{d-1} \sum_{k = 1}^d \left(\prod_{\substack{i=1\\ i \ne k}}^d \rho_0^0(t,x'_i-y_i)\right) 
\rho_{\gamma}^{0}(t,x'_k-y_k).
\end{eqnarray*}
\end{proof}

\begin{theorem}
\label{Holder}
For any $x, x', y \in \R^d$, $t \in (0,T]$ and $\gamma \in (0,\beta)$ we have
\begin{eqnarray}
\nonumber
&& |q(t,x,y) - q(t,x',y)| \le c \left( |x - x'|^{\beta - \gamma} \wedge 1\right)\\
\nonumber
&& \times \left[ \sum_{k = 1}^d \left[\prod_{\substack{i=1\\ i \ne k}}^d (t \rho_0^0(t,x_i-y_i))\right] 
\left(\rho_{\gamma}^0(t,x_k-y_k) + \rho_{\gamma - \beta}^{\beta}(t,x_k-y_k)\right) \right. \\
\label{Holderqin3}
&& + \left. \sum_{k = 1}^d \left[\prod_{\substack{i=1\\ i \ne k}}^d (t \rho_0^0(t,x'_i-y_i))\right] 
\left(\rho_{\gamma}^0(t,x'_k-y_k) + \rho_{\gamma - \beta}^{\beta}(t,x'_k-y_k)\right) \right].
\end{eqnarray}

\end{theorem}
\begin{proof}
By the definition of $q_n$, (\ref{newqn}) and Lemma \ref{q0Holder} we get for $n \in \N$
\begin{eqnarray}
\nonumber
&& |q_n(t,x,y) - q_n(t,x',y)| \\
\nonumber
&& \le \int_0^t \int_{\R^d} |q_0(t-s,x,z) - q_0(t-s,x',z)||q_{n-1}(s,z,y)| \, dz \, ds\\
\label{qn1a}
&&\le \frac{c C^{n-1}}{(n!)^{\beta/\alpha}} \left(|x-x'|^{\beta - \gamma} \wedge 1\right) \left(A(t,x,y) + A(t,x',y)\right),
\end{eqnarray}
where $C$ is the constant from (\ref{newqn}) and
\begin{eqnarray*}
&& A(t,x,y) \\
&& = \int_0^t \int_{\R^d} \sum_{k = 1}^d \left[\prod_{\substack{i=1\\ i \ne k}}^d  \rho_{\alpha}^0(t-s,x_i-z_i)\right] 
\left(\rho_{\gamma}^0(t-s,x_k-z_k) + \rho_{\gamma - \beta}^{\beta}(t-s,x_k-z_k)\right)\\
&& \times \sum_{m = 1}^d \left[\prod_{\substack{j=1\\ j \ne m}}^d \rho_{\alpha}^0(s,z_j-y_j)\right] 
\left(\rho_{\beta}^0(s,z_m-y_m) + \rho_{0}^{\beta}(s,z_m-y_m)\right) 
\, dz_1 \ldots dz_d \, ds.
\end{eqnarray*}

We have 
\begin{eqnarray}
\nonumber
A(t,x,y) &=& \sum_{k = 1}^d \left[\prod_{\substack{i=1\\ i \ne k}}^d t \rho_0^0(t,x_i-y_i)\right] B_k(t,x,y)\\
\nonumber
&& + \sum_{k = 1}^d \sum_{\substack{m=1\\ m \ne k}}^d \left[\prod_{\substack{i = 1\\ i \ne k, \, i \ne m}}^d t \rho_0^0(t,x_i-y_i)\right] \\
\label{Afor2}
&& \times [D_{k,m}(t,x,y) + E_{k,m}(t,x,y) + F_{k,m}(t,x,y) + G_{k,m}(t,x,y)],
\end{eqnarray}
where
\begin{eqnarray*}
B_k(t,x,y) &=&
\int_0^t \int_{\R} 
\left(\rho_{\gamma}^0(t-s,x_k-z_k) + \rho_{\gamma - \beta}^{\beta}(t-s,x_k-z_k)\right)\\
&& \times \left(\rho_{\beta}^0(s,z_k-y_k) + \rho_{0}^{\beta}(s,z_k-y_k)\right) \, dz_k \, ds,
\end{eqnarray*}
\begin{eqnarray*}
D_{k,m}(t,x,y) &=&
\int_0^t 
\int_{\R} \rho_{\gamma}^0(t-s,x_k-z_k) \rho_{\alpha}^{0}(s,z_k-y_k) \, dz_k\\
&& \times \int_{\R}  \rho_{\alpha}^0(t-s,x_m-z_m) \rho_{\beta}^0(s,z_m-y_m) \, dz_m \, ds,
\end{eqnarray*}
\begin{eqnarray*}
E_{k,m}(t,x,y) &=&
\int_0^t 
\int_{\R} \rho_{\gamma}^0(t-s,x_k-z_k)  \rho_{\alpha}^{0}(s,z_k-y_k) \, dz_k\\
&& \times \int_{\R}  \rho_{\alpha}^0(t-s,x_m-z_m) \rho_{0}^{\beta}(s,z_m-y_m) \, dz_m \, ds,
\end{eqnarray*}
\begin{eqnarray*}
F_{k,m}(t,x,y) &=&
\int_0^t 
\int_{\R} \rho_{\gamma - \beta}^{\beta}(t-s,x_k-z_k) \rho_{\alpha}^{0}(s,z_k-y_k) \, dz_k\\
&& \times \int_{\R}  \rho_{\alpha}^0(t-s,x_m-z_m) \rho_{\beta}^0(s,z_m-y_m) \, dz_m \, ds,
\end{eqnarray*}
\begin{eqnarray*}
G_{k,m}(t,x,y) &=&
\int_0^t 
\int_{\R} \rho_{\gamma-\beta}^{\beta}(t-s,x_k-z_k)  \rho_{\alpha}^{0}(s,z_k-y_k) \, dz_k\\
&& \times \int_{\R} \rho_{\alpha}^0(t-s,x_m-z_m) \rho_{0}^{\beta}(s,z_m-y_m) \, dz_m \, ds.
\end{eqnarray*}

By (\ref{conv2}) we get 
\begin{equation}
\label{Bfor}
B_k(t,x,y) \le 
c \left(\rho_{\gamma}^0(t,x_k-y_k) + \rho_{\gamma - \beta}^{\beta}(t,x_k-y_k)\right).
\end{equation}

Using (\ref{conv1}) we obtain 
\begin{eqnarray}
\nonumber
&& D_{k,m}(t,x,y)\\
\nonumber
&& \le c \int_0^t \left((t-s)^{\frac{\gamma-\alpha}{\alpha}} s + (t - s)^{\frac{\gamma}{\alpha}}\right)
\left(s^{\frac{\beta}{\alpha}} + (t - s) s^{\frac{\beta - \alpha}{\alpha}}\right) \, ds
\rho_{0}^0(t,x_k-y_k) \rho_{0}^0(t,x_m-y_m)\\
\nonumber
&& \le c t^{1 + \frac{\gamma}{\alpha} + \frac{\beta}{\alpha}} \rho_{0}^0(t,x_k-y_k) \rho_{0}^0(t,x_m-y_m)\\
\label{Dfor}
&& \le c t \rho_{\gamma}^0(t,x_k-y_k) \rho_{0}^0(t,x_m-y_m),
\end{eqnarray}
\begin{eqnarray}
\nonumber
&& E_{k,m}(t,x,y)\\
\nonumber
&& \le c \int_0^t \left((t-s)^{\frac{\gamma-\alpha}{\alpha}} s + (t - s)^{\frac{\gamma}{\alpha}}\right)
\left((t - s)^{\beta/\alpha} + (t - s) s^{\frac{\beta - \alpha}{\alpha}}\right) \, ds\\
\nonumber
&& \times
\rho_{0}^0(t,x_k-y_k) \rho_{0}^0(t,x_m-y_m)\\
\nonumber
&& + c \int_0^t \left((t-s)^{\frac{\gamma-\alpha}{\alpha}} s + (t - s)^{\frac{\gamma}{\alpha}}\right) \, ds
\rho_{0}^0(t,x_k-y_k) \rho_{0}^{\beta}(t,x_m-y_m)\\
\nonumber
&& \le c t^{1 + \frac{\gamma}{\alpha} + \frac{\beta}{\alpha}} \rho_{0}^0(t,x_k-y_k) \rho_{0}^0(t,x_m-y_m)
+ c t^{1 + \frac{\gamma}{\alpha}} \rho_{0}^0(t,x_k-y_k) \rho_{0}^{\beta}(t,x_m-y_m)\\
\label{Efor}
&& \le c t \rho_{\gamma}^0(t,x_k-y_k) \rho_{0}^0(t,x_m-y_m),
\end{eqnarray}
\begin{eqnarray}
\nonumber
&& F_{k,m}(t,x,y)\\
\nonumber
&& \le c \int_0^t \left((t-s)^{\frac{\gamma-\alpha}{\alpha}} s + (t - s)^{\frac{\gamma-\beta}{\alpha}} s^{\frac{\beta}{\alpha}}\right)
\left(s^{\frac{\beta}{\alpha}} + (t - s) s^{\frac{\beta - \alpha}{\alpha}}\right) \, ds\\
\nonumber
&& \times
\rho_{0}^0(t,x_k-y_k) \rho_{0}^0(t,x_m-y_m)\\
\nonumber
&& + c \int_0^t \left((t-s)^{\frac{\gamma-\beta}{\alpha}} \left( s^{\frac{\beta}{\alpha}} 
+ (t - s) s^{\frac{\beta - \alpha}{\alpha}}\right) \right) \, ds
\rho_{0}^{\beta}(t,x_k-y_k) \rho_{0}^{0}(t,x_m-y_m)\\
\nonumber
&& \le c t^{1 + \frac{\gamma}{\alpha} + \frac{\beta}{\alpha}} \rho_{0}^0(t,x_k-y_k) \rho_{0}^0(t,x_m-y_m)
+ c t^{1 + \frac{\gamma}{\alpha}} \rho_{0}^{\beta}(t,x_k-y_k) \rho_{0}^{0}(t,x_m-y_m)\\
\label{Ffor}
&& \le c t \rho_{\gamma}^0(t,x_k-y_k) \rho_{0}^0(t,x_m-y_m),
\end{eqnarray}
\begin{eqnarray}
\nonumber
&& G_{k,m}(t,x,y)\\
\nonumber
&& \le c \int_0^t \left((t-s)^{\frac{\gamma-\alpha}{\alpha}} s + (t - s)^{\frac{\gamma-\beta}{\alpha}} s^{\frac{\beta}{\alpha}}\right)
\left((t-s)^{\frac{\beta}{\alpha}} + (t - s) s^{\frac{\beta - \alpha}{\alpha}}\right) \, ds\\
&& \nonumber \times
\rho_{0}^0(t,x_k-y_k) \rho_{0}^0(t,x_m-y_m)\\
\nonumber
&& + c \int_0^t \left((t-s)^{\frac{\gamma-\alpha}{\alpha}} s +  (t-s)^{\frac{\gamma-\beta}{\alpha}} s^{\frac{\beta}{\alpha}} \right) \, ds
\rho_{0}^{0}(t,x_k-y_k) \rho_{0}^{\beta}(t,x_m-y_m)\\
\nonumber
&& + c \int_0^t \left((t-s)^{\frac{\gamma-\beta}{\alpha}} \left( (t-s)^{\frac{\beta}{\alpha}} + (t - s) s^{\frac{\beta - \alpha}{\alpha}}\right) \right) \, ds
\rho_{0}^{\beta}(t,x_k-y_k) \rho_{0}^{0}(t,x_m-y_m)\\
\nonumber
&& + c \int_0^t (t-s)^{\frac{\gamma-\beta}{\alpha}} \, ds
\rho_{0}^{\beta}(t,x_k-y_k) \rho_{0}^{\beta}(t,x_m-y_m)\\
\nonumber
&& \le c t^{1 + \frac{\gamma}{\alpha} + \frac{\beta}{\alpha}} \rho_{0}^0(t,x_k-y_k) \rho_{0}^0(t,x_m-y_m)
+ c t^{1 + \frac{\gamma}{\alpha}} \rho_{0}^{0}(t,x_k-y_k) \rho_{0}^{\beta}(t,x_m-y_m)\\
\nonumber
&& + c t^{1 + \frac{\gamma}{\alpha}} \rho_{0}^{\beta}(t,x_k-y_k) \rho_{0}^0(t,x_m-y_m)
+ c t^{1 + \frac{\gamma}{\alpha} - \frac{\beta}{\alpha}} \rho_{0}^{\beta}(t,x_k-y_k) \rho_{0}^{\beta}(t,x_m-y_m)\\
\label{Gfor}
&& \le c t \rho_{\gamma}^0(t,x_k-y_k) \rho_{0}^0(t,x_m-y_m)
+ c t \rho_{\gamma - \beta}^{\beta}(t,x_k-y_k) \rho_{0}^0(t,x_m-y_m).
\end{eqnarray}

By (\ref{Afor2}-\ref{Gfor}) we obtain
$$
A(t,x,y) \le c \sum_{k = 1}^d \left(\prod_{\substack{i=1\\ i \ne k}}^d t \rho_0^0(t,x_i - y_i)\right) 
\left(\rho_{\gamma}^0(t,x_k-y_k) + \rho_{\gamma - \beta}^{\beta}(t,x_k-y_k)\right).
$$
Using this and (\ref{qn1a}) we obtain that for any $n \in \N$, $x, x', y \in \R^d$, $t \in (0,T]$ and $\gamma \in (0,\beta)$ 
\begin{eqnarray*}
&& |q_n(t,x,y) - q_n(t,x',y)| \le \frac{c C^{n-1}}{(n!)^{\beta/\alpha}} \left( |x - x'|^{\beta - \gamma} \wedge 1\right)\\
&& \times \left[ \sum_{k = 1}^d \left[\prod_{\substack{i=1\\ i \ne k}}^d (t \rho_0^0(t,x_i-y_i))\right] 
\left(\rho_{\gamma}^0(t,x_k-y_k) + \rho_{\gamma - \beta}^{\beta}(t,x_k-y_k)\right) \right. \\
&& + \left. \sum_{k = 1}^d \left[\prod_{\substack{i=1\\ i \ne k}}^d (t \rho_0^0(t,x'_i-y_i))\right] 
\left(\rho_{\gamma}^0(t,x'_k-y_k) + \rho_{\gamma - \beta}^{\beta}(t,x'_k-y_k)\right) \right].
\end{eqnarray*}
This, Lemma \ref{q0Holder} and the definition of $q$ imply the assertion of the theorem.
\end{proof}

\begin{lemma}
\label{pyHolder}
For all $\gamma \in (0,1]$, $x, x', y \in \Rd$, $t >0$, we have
$$
|p_y(t,x) - p_y(t,x')| \le c |x - x'|^{\gamma} t^{-\gamma/\alpha} \left(\left[\prod_{i = 1}^d g_t(x_i)\right] + \left[\prod_{i = 1}^d g_t(x'_i)\right]\right).
$$
\end{lemma}
\begin{proof}
By Lemma \ref{productgt1} we get
$$
|p_y(t,x) - p_y(t,x')| \le c \left(|x - x'| t^{-1/\alpha} \wedge 1\right) \left(\left[\prod_{i = 1}^d g_t(x_i)\right] + \left[\prod_{i = 1}^d g_t(x'_i)\right]\right).
$$
Since $\left(|x - x'| t^{-1/\alpha} \wedge 1\right) \le  |x - x'|^{\gamma} t^{-\gamma/\alpha}$ we obtain the assertion of the lemma.
\end{proof}

By Lemma \ref{productgt1} and the formula for $p_y(t,x)$ we obtain
\begin{lemma}
\label{lemmagradient}
For any $x, y \in \Rd$ and $t >0$ we have
$$
\left|\nabla p_y(t,\cdot)(x-y) \right| \le c t^{-\frac{1}{\alpha}} \prod_{i = 1}^d \rho_{\alpha}^0(t,x_i-y_i).
$$
\end{lemma} 

\begin{lemma}
\label{integralgrad}
For any $x \in \Rd$ and $t \in (0,T]$ we have
$$
\left|\int_{\Rd} \nabla p_y(t,\cdot)(x-y) \, dy\right| \le c t^{\frac{\beta-1}{\alpha}}.
$$
\end{lemma} 
\begin{proof}
Let $D_1 p_y(t,\cdot)(x-y) = \lim_{h \to 0} (p_y(t,x-y+he_1)-p_y(t,x-y))/h$. It is enough to prove the estimate for 
$\text{I} = \int_{\Rd} D_1 p_y(t,\cdot)(x-y) \, dy$. Let $\gamma \in (0,\beta)$ and put $\tilde{y} = (x_1,y_2,\ldots,y_d)$. We have
\begin{eqnarray*}
|\text{I}| &=& 
\left| \int_{\Rd} \left[ \frac{1}{a_{11}^2(y)}g'_t\left(\frac{x_1-y_1}{a_{11}(y)}\right)
\left[\prod_{i=2}^d \frac{1}{a_{ii}(y)}g_t\left(\frac{x_i-y_i}{a_{ii}(y)}\right)\right]
\right. \right. \\
&& \left. \left. - \frac{1}{a_{11}^2(\tilde{y})}g'_t\left(\frac{x_1-y_1}{a_{11}(\tilde{y})}\right)
\left[\prod_{i=2}^d \frac{1}{a_{ii}(\tilde{y})}g_t\left(\frac{x_i-y_i}{a_{ii}(\tilde{y})}\right)\right] \right] \, dy \right|\\
&\le& \left| \int_{\R^{d-1}} \left[ \int_{\R} \left[\frac{1}{a_{11}^2(y)}g'_t\left(\frac{x_1-y_1}{a_{11}(y)}\right)
- \frac{1}{a_{11}^2(\tilde{y})}g'_t\left(\frac{x_1-y_1}{a_{11}(\tilde{y})}\right) \right] \, dy_1 \right] \right.\\
&& \left. \times \left[\prod_{i=2}^d \frac{1}{a_{ii}(\tilde{y})}g_t\left(\frac{x_i-y_i}{a_{ii}(\tilde{y})}\right)\right] \,dy_2, \, \ldots \,dy_d \right|\\
&+& \left| \int_{\Rd} \frac{1}{a_{11}^2(y)}g'_t\left(\frac{x_1-y_1}{a_{11}(y)}\right) \right. \\
&& \left. \times \left[\prod_{i=2}^d \frac{1}{a_{ii}(y)}g_t\left(\frac{x_i-y_i}{a_{ii}(y)}\right)
- \prod_{i=2}^d \frac{1}{a_{ii}(\tilde{y})}g_t\left(\frac{x_i-y_i}{a_{ii}(\tilde{y})}\right)\right] \, dy \right|\\
&=& \text{II} + \text{III}.
\end{eqnarray*}
By \cite[(2.31)]{CZ} we get 
$$
\left|\frac{1}{a_{11}^2(y)}g'_t\left(\frac{x_1-y_1}{a_{11}(y)}\right)
- \frac{1}{a_{11}^2(\tilde{y})}g'_t\left(\frac{x_1-y_1}{a_{11}(\tilde{y})}\right) \right|
\le c (|x_1-y_1|^{\beta} \wedge 1) t^{-1/\alpha} (\rho_{\alpha}^{0}(t,x_1-y_1) + \rho_{\alpha - \gamma}^{\gamma}(t,x_1-y_1)).
$$
Using this and (\ref{conv0}) we obtain $\text{II} \le c t^{\frac{\beta-1}{\alpha}}$. Note also that 
$$
\left|\frac{1}{a_{11}^2(y)}g'_t\left(\frac{x_1-y_1}{a_{11}(y)}\right)\right| \le c \rho_{\alpha-1}^{0}(t,x_1-y_1).
$$
Using this, Corollary \ref{mintegrallemma} and (\ref{conv0}) we get $\text{III} \le c t^{\frac{\beta-1}{\alpha}}$.
\end{proof}

Similarly as in \cite{CZ} we denote
$$
\phi_y(t,x,s) = \int_{\Rd} p_z(t-s,x-z)q(s,z,y) \, dz
$$
and
$$
\varphi_y(t,x) = \int_0^t  \phi_y(t,x,s) \, ds.
$$
Clearly we have 
\begin{equation} 
\label{pA}
p^A(t,x,y) = p_y(t,x-y) + \varphi_y(t,x).
\end{equation}

By well known estimates of $\nabla p_z(t-s,\cdot)(x-z)$ and Theorem \ref{thmq} we easily obtain the following result.
\begin{lemma}
\label{varphiintegral}
For any $x, y \in \Rd$, $t > 0$ and $s \in (0,t)$ we have
$$
\nabla_x \phi_y(t,x,s) = 
\int_{\Rd} \nabla p_z(t-s,\cdot)(x-z) q(s,z,y) \, dz.
$$
\end{lemma}

The next result is the most important step in proving gradient estimates of $p^A(t,x,y)$.
\begin{lemma}
\label{gradientvarphi}
For any $\alpha \in (1,2)$, $x, y \in \Rd$ and $t \in (0,T]$ we have
\begin{equation}
\label{gradientvarphi1}
\nabla_x \varphi_y(t,x) = 
\int_0^t \int_{\Rd} \nabla p_z(t-s,\cdot)(x-z) q(s,z,y) \, dz \, ds.
\end{equation}
and
\begin{equation}
\label{gradientvarphi2}
|\nabla_x \varphi_y(t,x)| \le c t^{-\frac{1}{\alpha}} \prod_{i = 1}^d \rho_{\alpha}^0(t,x_i-y_i).
\end{equation}
\end{lemma}
\begin{proof}

 Let $x, y \in \Rd$, $t \in (0,T]$ and $s \in (0,t)$. The main tool used in this case is Theorem \ref{thmq}. Using this theorem, Lemmas \ref{lemmagradient}, \ref{varphiintegral} and (\ref{conv1}) we obtain
\begin{eqnarray*}
&& |\nabla_x \phi_y(t,x,s)| \le \int_{\Rd} \left|\nabla p_z(t-s,\cdot)(x-z)\right| \left|q(s,z,y)\right| \, dz\\ 
&& \le c \int_{\Rd} (t-s)^{-1/\alpha} p_z(t-s,x-z)\\ 
&& \times \sum_{m=1}^d \left[\prod_{\substack{i=1\\ i \ne m}}^d \rho_{\alpha}^0(s,z_i-y_i)\right]
[ \rho_{0}^{\beta}(s,z_m-y_m) + \rho_{\beta}^0(s,z_m-y_m)] \, dz\\
&& \le c \sum_{m=1}^d \left[\prod_{\substack{i=1\\ i \ne m}}^d \rho_{\alpha}^0(t,x_i-y_i)\right]
 \int_{\R} \rho_{\alpha-1}^0(t-s,x_m-z_m) \\
&& \times [ \rho_{0}^{\beta}(s,z_m-y_m) +\rho_{\beta}^0(s,z_m-y_m)] \, dz\\
&& \le c \sum_{m=1}^d \left[\prod_{\substack{i=1\\ i \ne m}}^d \rho_{\alpha}^0(t,x_i-y_i)\right]
\left[
(t - s)^{\frac{\beta-1}{\alpha}} \rho_{0}^{0}(t,x_m-y_m) \right.\\
&& 
+ (t - s)^{\frac{\alpha-1}{\alpha}} s^{\frac{\beta-\alpha}{\alpha}}\rho_{0}^{0}(t,x_m-y_m) 
 + (t - s)^{-\frac{1}{\alpha}} \rho_{0}^{\beta}(t,x_m-y_m)\\
&& \left. + (t - s)^{-\frac{1}{\alpha}} s^{\frac{\beta}{\alpha}}\rho_{0}^{0}(t,x_m-y_m)
\right].
\end{eqnarray*}
It follows that 
$$
 \nabla_x \left[\int_0^t  \phi_y(t,x,s) \, ds \right] = \int_0^t \nabla_x  \phi_y(t,x,s) \, ds,
$$
which implies (\ref{gradientvarphi1}). We also obtain
\begin{eqnarray*}
&&\left|\int_0^t \nabla_x  \phi_y(t,x,s) \, ds\right| \\
&& \le c \sum_{m=1}^d \left[\prod_{\substack{i=1\\ i \ne m}}^d \rho_{\alpha}^0(t,x_i-y_i)\right]
\int_0^t \left[
(t - s)^{\frac{\beta-1}{\alpha}} \rho_{0}^{0}(t,x_m-y_m) \right. \\
&& + (t - s)^{\frac{\alpha-1}{\alpha}} s^{\frac{\beta-\alpha}{\alpha}}\rho_{0}^{0}(t,x_m-y_m) 
+ (t - s)^{-\frac{1}{\alpha}} \rho_{0}^{\beta}(t,x_m-y_m) \\
&& + \left. (t - s)^{-\frac{1}{\alpha}} s^{\frac{\beta}{\alpha}}\rho_{0}^{0}(t,x_m-y_m) \right] \, ds.\\
&& \le c \sum_{m=1}^d \left[\prod_{\substack{i=1\\ i \ne m}}^d \rho_{\alpha}^0(t,x_i-y_i)\right]
[ \rho_{\alpha+\beta-1}^{0}(t,x_m-y_m) + \rho_{\alpha-1}^{\beta}(t,x_m-y_m)],
\end{eqnarray*}
which implies (\ref{gradientvarphi2}).




\end{proof}

\begin{proposition}
\label{gradient1} For any $\alpha \in (1,2)$, $t \in (0,T]$ and $x, y \in \Rd$ we have
$$
|\nabla_x p^A(t,x,y)| \le c t^{-\frac1\alpha} \prod_{i = 1}^d g_t(x_i - y_i).
$$
\end{proposition}
\begin{proof}
The assertion follows from formula (\ref{pA}) and Lemmas \ref{lemmagradient}, \ref{gradientvarphi}.
\end{proof}

\section{Feller semigroup}
For any bounded Borel $f: \Rd \to \R$, $t \in (0,\infty)$ and $x \in \Rd$ we define
$$
P_t^A f(x) = \int_{\Rd} p^A(t,x,y) f(y) \, dy.
$$
The main aim of this section is to show that $\{P_t^A\}$ is a Feller semigroup. 

For any $\eps \ge 0$ and $x \in \Rd$ we put
$$
\calL_{\eps} f(x) = \frac{\calA_{\alpha}}{2} \sum_{i = 1}^d \int_{\{z_i:\, |z_i| > \eps\}} \delta_f(x,e_i z_i) \, \s_i(x) \frac{dz_i}{|z_i|^{1 + \alpha}},
$$
$$
\calL_{\eps}^y f(x) = \frac{\calA_{\alpha}}{2} \sum_{i = 1}^d \int_{\{z_i:\, |z_i| > \eps\}} \delta_f(x,e_i z_i) \, \s_i(y) \frac{dz_i}{|z_i|^{1 + \alpha}}.
$$


Using \cite [(3.13)]{CZ} and the same arguments as in the proof of Lemma \ref{integralgrad} we obtain
\begin{lemma}
\label{integralfractional}
For any $\eps > 0$, $x, y \in \Rd$ and $t \in (0,T]$ we have
$$
\left|\int_{\Rd} \calL_{\eps}^x p_y(t,\cdot)(x-y) \, dy\right| \le c t^{\frac{\beta-\alpha}{\alpha}}.
$$
\end{lemma} 

\begin{lemma}
\label{fractionalheat}
For any $x, y \in \Rd$ and $t > 0$ $\calL^x \varphi_y(t,x)$ is well defined and we have
\begin{equation}
\label{genvarphi1}
\calL^x \varphi_y(t,x) =
\int_0^t \int_{\Rd} \calL^x p_z(t-s,\cdot)(x-z) q(s,z,y) \, dz \, ds.
\end{equation}
Fix $\gamma \in (0,\beta)$. There exists $c$ such that for any $\eps > 0$,   $t \in (0,T]$ and $x, y \in \Rd$ we have
\begin{equation}
\label{epsheat}
|\calL_{\eps} p^A(t,\cdot,y)(x)| \le c t^{\frac{-\alpha+\gamma-\beta}{\alpha}} \prod_{i = 1}^d \rho_{\alpha}^0(t,x_i-y_i).
\end{equation}
Moreover, $t \to \calL^x \varphi_y(t,x)$ is continuous on $(0,T)$ for any $x, y \in \Rd$.
\end{lemma}
\begin{proof}Let $\eps > 0$. We have
$$
\calL_{\eps} \varphi_y(t,x) = \int_0^{t} \int_{\Rd} \calL_{\eps}^x p_z(t-s,\cdot)(x-z) q(s,z,y) \, dz \, ds.
$$
By Lemma \ref{generatorpy1} and Theorem \ref{thmq} one easily gets
\begin{equation}
\label{changeoflimit1}
\lim_{\eps \to 0^+} \int_{\Rd} \calL_{\eps}^x p_z(t-s,\cdot)(x-z) q(s,z,y) \, dz  =
\int_{\Rd} \calL^x p_z(t-s,\cdot)(x-z) q(s,z,y) \, dz.
\end{equation}
The most difficult part of the proof is to justify
\begin{eqnarray}
\nonumber
&& \lim_{\eps \to 0^+} \int_0^{t} \int_{\Rd} \calL_{\eps}^x p_z(t-s,\cdot)(x-z) q(s,z,y) \, dz \, ds \\
\label{changeoflimit2}
&& = \int_0^{t} \lim_{\eps \to 0^+} \int_{\Rd} \calL_{\eps}^x p_z(t-s,\cdot)(x-z) q(s,z,y) \, dz \, ds.
\end{eqnarray}
We have
\begin{eqnarray*}
\calL_{\eps} \varphi_y(t,x) 
&=& \int_0^{t/2} \int_{\Rd} \calL_{\eps}^x p_z(t-s,\cdot)(x-z) q(s,z,y) \, dz \, ds\\
&+& \int_{t/2}^{t}  \int_{\Rd} \calL_{\eps}^x p_z(t-s,\cdot)(x-z) \, dz q(s,x,y) \, ds\\
&+&  \int_{t/2}^{t} \int_{\Rd}  \calL_{\eps}^x p_z(t-s,\cdot)(x-z)  (q(s,z,y)-q(s,x,y)) \, dz \, ds\\
&=& D(t,x,y)+E(t,x,y)+F(t,x,y).
\end{eqnarray*}
For $s \in (0,t/2)$ by Theorem \ref{thmq}, Lemma \ref{generatorpy1} and  (\ref{conv1}) we obtain 
\begin{eqnarray*} 
&&\int_{\Rd} \left|\calL_{\eps}^x p_z(t-s,\cdot)(x-z) q(s,z,y) \right|\, dz\\
&&\le c \sum_{m=1}^d \left[\prod_{\substack{i=1\\ i \ne m}}^d \rho_{\alpha}^0(t,x_i-y_i)\right]\\
&&  \times
\int_{\R} \rho_{0}^0(t-s,x_m-z_m) [ \rho_{0}^{\beta}(s,z_m-y_m) +\rho_{\beta}^0(s,z_m-y_m)] \, dz\\
&& \le c \sum_{m=1}^d \left[\prod_{\substack{i=1\\ i \ne m}}^d \rho_{\alpha}^0(t,x_i-y_i)\right]\\
&& \times  \left[(t - s)^{\frac{\beta-\alpha}{\alpha}}\rho_0^0(t,x_m-y_m) + s^{\frac{\beta-\alpha}{\alpha}}\rho_0^0(t,x_m-y_m) 
+ (t - s)^{-1}\rho_0^{\beta}(t,x_m-y_m)\right].
\end{eqnarray*}
It follows that 
\begin{eqnarray}
\nonumber
&& \lim_{\eps \to 0^+} \int_0^{t/2} \int_{\Rd} \calL_{\eps}^x p_z(t-s,\cdot)(x-z) q(s,z,y) \, dz \, ds\\ 
\label{D1formula} 
&& = \int_0^{t/2} \lim_{\eps \to 0^+} \int_{\Rd} \calL_{\eps}^x p_z(t-s,\cdot)(x-z) q(s,z,y) \, dz \, ds.
\end{eqnarray}
and
\begin{equation}
\label{D2formula}
D(t,x,y) \le c t^{-1} \prod_{i = 1}^d \rho_{\alpha}^0(t,x_i-y_i).
\end{equation}

For $s \in (t/2,t)$ by Theorem \ref{thmq} and Lemma \ref{integralfractional} we obtain
\begin{eqnarray*}
 \left| \int_{\Rd} \calL_{\eps}^x p_z(t-s,\cdot)(x-z) \, dz q(s,x,y) \right|
&\le& c  (t-s)^{{\frac{\beta-\alpha}{\alpha}}}
t^{-1} \prod_{i = 1}^d \rho_{\alpha}^0(t,x_i-y_i).
\end{eqnarray*}
It follows that 
\begin{eqnarray}
\nonumber
&&\lim_{\eps \to 0^+} \int_{t/2}^{t}  \int_{\Rd} \calL_{\eps}^x p_z(t-s,\cdot)(x-z) \, dz q(s,x,y) \, ds\\
\label{E1formula} 
&& = \int_{t/2}^{t} \lim_{\eps \to 0^+} \int_{\Rd} \calL_{\eps}^x p_z(t-s,\cdot)(x-z) \, dz q(s,x,y) \, ds.
\end{eqnarray}
and
\begin{equation}
\label{E2formula}
E(t,x,y) \le c t^{-1} \prod_{i = 1}^d \rho_{\alpha}^0(t,x_i-y_i).
\end{equation}

Now, we need to obtain some estimates which will be crucial in studying the most difficult term $F(t,x,y)$.
By Lemma \ref{generatorpy1} and Theorem \ref{Holder} there exists $c$ (not depending on $\eps$) such that for any $t \in (0,T]$ and $x, y \in \Rd$ we have
\begin{eqnarray*}
&&\int_{\Rd} \sum_{i=1}^d \left[
\int_{|w_i| > \eps} \left|\delta_{p_z}(t-s,x-z,w_i e_i) \right| |w_i|^{-1-\alpha} \, dw_i \left|q(s,z,y)  - q(s,x,y)\right| 
\right] dz\\
&\le& c  (t-s)^{-1} \int_{\Rd}  (|x-z|^{\beta - \gamma} \wedge 1) \left[\prod_{i = 1}^d \rho_{\alpha}^0(t-s,x_i-z_i)\right] \, dz  
\sum_{m=1}^d \left[\prod_{\substack{i=1\\ i \ne m}}^d \rho_{\alpha}^0(s,x_i-y_i)\right]\\
&\times&  [ \rho_{\gamma-\beta}^{\beta}(s,x_m-y_m) +\rho_{\gamma}^0(s,x_m-y_m)]\\
&+& c (t-s)^{-1} \int_{\Rd}  (|x-z|^{\beta - \gamma} \wedge 1) \left[\prod_{i = 1}^d \rho_{\alpha}^0(t-s,x_i-z_i)\right] 
\sum_{m=1}^d \left[\prod_{\substack{i=1\\ i \ne m}}^d \rho_{\alpha}^0(s,z_i-y_i)\right]\\
&\times&  [ \rho_{\gamma-\beta}^{\beta}(s,z_m-y_m) +\rho_{\gamma}^0(s,z_m-y_m)] \, dz\\
&=& B_1(s,t,x,y) + B_2(s,t,x,y).
\end{eqnarray*}
Clearly $(|x-z|^{\beta - \gamma} \wedge 1) \le \sum_{k=1}^d (|x_k-z_k|^{\beta - \gamma} \wedge 1)$. It follows that
\begin{eqnarray*}
&&(t-s)^{-1} \int_{\Rd}  (|x-z|^{\beta - \gamma} \wedge 1) \left[\prod_{i = 1}^d \rho_{\alpha}^0(t-s,x_i-z_i)\right] \, dz\\  
&\le& c \sum_{k=1}^d \int_{\R} \rho_{0}^{\beta -\gamma}(t-s,x_k-z_k) \, dz_k\\
&\le& c (t-s)^{\frac{\beta - \gamma -\alpha}{\alpha}}.
\end{eqnarray*}
Hence
\begin{eqnarray}
\nonumber
B_1(s,t,x,y)
&\le& c (t-s)^{\frac{\beta - \gamma -\alpha}{\alpha}}
\sum_{m=1}^d \left[\prod_{\substack{i=1\\ i \ne m}}^d \rho_{\alpha}^0(s,x_i-y_i)\right]\\
\label{B1inequality}
&\times&  [ \rho_{\gamma-\beta}^{\beta}(s,x_m-y_m) +\rho_{\gamma}^0(s,x_m-y_m)].
\end{eqnarray}
We also have
\begin{eqnarray*}
B_2(s,t,x,y)
&\le& c \sum_{m=1}^d \int_{\Rd} \left[\prod_{\substack{i=1\\ i \ne m}}^d \rho_{\alpha}^0(t-s,x_i-z_i) \rho_{\alpha}^0(s,z_i-y_i)\right]\\
&\times&   \rho_{0}^{\beta-\gamma}(t-s,x_m-z_m) 
[ \rho_{\gamma-\beta}^{\beta}(s,z_m-y_m) +\rho_{\gamma}^0(s,z_m-y_m)] \, dz\\
&+& c \sum_{m=1}^d \sum_{\substack{k=1\\ k \ne m}}^d
\int_{\Rd} \left[\prod_{\substack{i=1\\ i \ne m,k}}^d \rho_{\alpha}^0(t-s,x_i-z_i) \rho_{\alpha}^0(s,z_i-y_i)\right]\\
&\times&   \rho_{\alpha}^{0}(t-s,x_m-z_m) 
[ \rho_{\gamma-\beta}^{\beta}(s,z_m-y_m) +\rho_{\gamma}^0(s,z_m-y_m)]\\
&\times&   \rho_{0}^{\beta-\gamma}(t-s,x_k-z_k) \rho_{\alpha}^0(s,z_k-y_k) \, dz\\
&=& B_3(s,t,x,y) + B_4(s,t,x,y).
\end{eqnarray*}
By (\ref{conv1}),  we have 
\begin{eqnarray}
\nonumber
&& B_3(s,t,x,y)\\
\nonumber
&\le& c \sum_{m=1}^d \left[\prod_{\substack{i=1\\ i \ne m}}^d \rho_{\alpha}^0(t,x_i-y_i)\right]
\left[(t - s)^{\frac{2\beta-\gamma -\alpha}{\alpha}} s^{\frac{\gamma-\beta}{\alpha}}\rho_0^0(t,x_m-y_m) + s^{\frac{\beta-\alpha}{\alpha}}\rho_0^0(t,x_m-y_m) \right.\\
\nonumber
&& + (t - s)^{\frac{\beta-\gamma -\alpha}{\alpha}} s^{\frac{\gamma-\beta}{\alpha}}\rho_0^{\beta}(t,x_m-y_m)
+ s^{\frac{\gamma-\alpha}{\alpha}}\rho_0^{\beta - \gamma}(t,x_m-y_m)\\
\label{B3inequality}
&& \left.
+ (t - s)^{\frac{\beta-\gamma -\alpha}{\alpha}} s^{\frac{\gamma}{\alpha}}\rho_0^0(t,x_m-y_m)\right].
\end{eqnarray}
By (\ref{conv1}), we also have
\begin{eqnarray}
\nonumber
&& B_4(s,t,x,y)\\
\nonumber
&\le& c \sum_{m=1}^d \sum_{\substack{k=1\\ k \ne m}}^d
\left[\prod_{\substack{i=1\\ i \ne m,k}}^d \rho_{\alpha}^0(t,x_i-y_i)\right]\\
\nonumber
&\times&\left[\left[
(t-s)^{\frac{\beta}{\alpha}} s^{\frac{\gamma - \beta}{\alpha}}
+ (t-s)^{\frac{\alpha}{\alpha}} s^{\frac{\gamma - \alpha}{\alpha}}
+s^{\frac{\gamma}{\alpha}}
\right]\rho_{0}^{0}(t,x_m-y_m) 
+ s^{\frac{\gamma-\beta}{\alpha}} \rho_{0}^{\beta}(t,x_m-y_m) 
\right]\\
\label{B4inequality}
&\times&\left[\left[
(t-s)^{\frac{\beta-\gamma-\alpha}{\alpha}} s^{\frac{\alpha}{\alpha}}
+s^{\frac{\beta-\gamma}{\alpha}}
\right]\rho_{0}^{0}(t,x_k-y_k) 
+ \rho_{0}^{\beta-\gamma}(t,x_k-y_k) 
\right]
\end{eqnarray}

By (\ref{B1inequality}-\ref{B4inequality}) and (\ref{conv0}),  we get 
\begin{eqnarray}
&&\lim_{\eps \to 0^+} \int_{t/2}^{t} \int_{\Rd}  \calL_{\eps}^x p_z(t-s,\cdot)(x-z)  (q(s,z,y)-q(s,x,y)) \, dz \, ds\nonumber\\
&&= \int_{t/2}^{t} \lim_{\eps \to 0^+} \int_{\Rd}  \calL_{\eps}^x p_z(t-s,\cdot)(x-z)  (q(s,z,y)-q(s,x,y)) \, dz \, ds. \label{F1formula}
\end{eqnarray}
and
\begin{equation}
\label{F2formula}
F(t,x,y) \le c t^{\frac{-\alpha+\gamma-\beta}{\alpha}} \prod_{i = 1}^d \rho_{\alpha}^0(t,x_i-y_i).
\end{equation} 

By (\ref{D1formula}), (\ref{E1formula}), (\ref{F1formula}) we get (\ref{changeoflimit2}). We also get continuity $t \to \calL^x \varphi_y(t,x)$. By (\ref{changeoflimit1}) and (\ref{changeoflimit2}) we obtain (\ref{genvarphi1}).
Using (\ref{D2formula}), (\ref{E2formula}), (\ref{F2formula}), Lemma \ref{generatorpy1} and formula (\ref{pA}) we get (\ref{epsheat}).
\end{proof}

The next result is an analogue of \cite[Theorem 4.1]{CZ}. Its proof is almost the same as the proof of \cite[Theorem 4.1]{CZ} and is omitted.
\begin{proposition}
\label{maximumprinciple}
Let $u(t,x) \in C_b([0,T]\times\Rd)$ with
\begin{equation}
\label{mp1}
\lim_{t \to 0^+} \sup_{x \in \Rd} |u(t,x) - u(0,x)| = 0.
\end{equation}
Assume that 
\begin{equation}
\label{mp2}
t \to \calL u(t,x) \quad \text{is continuous on} \quad (0,T] \quad \text{for each} \quad x \in \Rd
\end{equation}
and for any $\eps \in (0,1)$ and some $\gamma \in ((\alpha - 1)\vee 0,1)$
\begin{equation}
\label{mp3}
\sup_{t \in (\eps,T)} |u(t,x) - u(t,x')| \le K_{\eps} |x - x'|^{\gamma}, \quad x,x' \in \Rd.
\end{equation}
If $u$ satisfies
\begin{equation}
\label{mp4}
\frac{\partial}{\partial t} u(t,x) = \calL u(t,x), \quad t \in (0,T], \, x \in \Rd,
\end{equation}
then 
\begin{equation}
\label{mp5}
\sup_{t \in (0,T)} \sup_{x \in \Rd} u(t,x) \le \sup_{x \in \Rd} u(0,x).
\end{equation}
\end{proposition}

\begin{lemma}
\label{interchange}
Let $x,y \in \R^d$. Then the mapping $t \to \varphi_y(t,x)$ is absolutely continuous on $(0,T]$. For any $t \in (0,T)$ we have
\begin{equation}
\label{tvarphi}
\frac{\partial \varphi_y}{\partial t}(t,x) = q(t,x,y) +
\int_0^t \int_{\Rd} \calL^z p_z(t-s,\cdot)(x-z) q(s,z,y) \, dz \, ds.
\end{equation}
\end{lemma}
\begin{proof}
Let $h > 0$ be such that $t + h < T$. We have
\begin{eqnarray*}
&&\frac{\varphi_y(t+h,x)-\varphi_y(t,x)}{h}\\
&& = \frac{1}{h} \int_0^{t+h} \int_{\Rd} p_z(t+h-s,x-z) q(s,z,y) \, dz \, ds\\ 
&& - \frac{1}{h} \int_0^{t} \int_{\Rd} p_z(t-s,x-z) q(s,z,y) \, dz \, ds\\
&& = \frac{1}{h} \int_0^{t+h} \int_{\Rd} p_z(t+h-s,x-z) q(s,z,y) \, dz \, ds\\ 
&& - \frac{1}{h} \int_0^{t} \int_{\Rd} p_z(t+h-s,x-z) q(s,z,y) \, dz \, ds\\
&& + \frac{1}{h} \int_0^{t} \int_{\Rd} p_z(t+h-s,x-z) q(s,z,y) \, dz \, ds\\ 
&& - \frac{1}{h} \int_0^{t} \int_{\Rd} p_z(t-s,x-z) q(s,z,y) \, dz \, ds\\
&& = \frac{1}{h} \int_t^{t+h} \int_{\Rd} p_z(t+h-s,x-z) q(s,z,y) \, dz \, ds\\
&& + \int_0^{t} \int_{\Rd} \frac{p_z(t+h-s,x-z)-p_z(t-s,x-z)}{h} q(s,z,y) \, dz \, ds\\
&& = \text{I} + \text{II}. 
\end{eqnarray*}

After change of variables $t+h-s=u$ we have

\begin{eqnarray*}\text{I}&=& \frac{1}{h} \int_0^{h} \int_{\Rd} p_z(u,x-z) q(t+h-u,z,y) \, dz \, du\\&= &
\frac{1}{h} \int_0^{h} \int_{\Rd} (p_z(u,x-z)-p_x(u,x-z)) q(t+h-u,z,y) \, dz \, du \\&+ &
\frac{1}{h} \int_0^{h} \int_{\Rd} p_x(u,x-z) q(t+h-u,z,y) \, dz \, du\\&= &\text{I}_1 + \text{I}_2. \end{eqnarray*}

By Theorem \ref{thmq} $,\sup_{u\le h, z\in \R^d } q(t+h-u,z,y)\le M<\infty$. Moreover,  from Corollary \ref{productgt2}, 
$$|p_z(u,x-z)-p_x(u,x-z)| \le c p_x(u,x-z) (|x-z|^\beta \wedge 1).$$
Hence, 
\begin{equation}
\label{I1}
\limsup_{h \to 0^+} |\text{I}_1| \le cM \limsup_{h \to 0^+} \frac{1}{h} \int_0^{h} \int_{\Rd} p_x(u,x-z)(|x-z|^\beta \wedge 1)  \, dz \, du=0 .
\end{equation}
 Next, by  Theorem \ref{thmq}, the function  $q(s,z,y)$ is  continuous and bounded on $[t, T]\times \R^d$, as a function of $s$ and $z$. Since the measures $\mu_u(dz)= p_x(u,x-z)dz$ converge weakly to $\delta_x$ as $u \to 0^+$,  we obtain 
%
\begin{equation}
\label{I}
\lim_{h \to 0^+} \text{I}_2 = q(t,x,y).
\end{equation}
We have
\begin{eqnarray*}
\text{II} &=&
\int_0^{t/2} \int_{\Rd} \frac{p_z(t+h-s,x-z)-p_z(t-s,x-z)}{h} q(s,z,y) \, dz \, ds\\
&+& \int_{t/2}^{t} \int_{\Rd} \frac{p_z(t+h-s,x-z)-p_z(t-s,x-z)}{h} q(s,z,y) \, dz \, ds\\
&=& \text{III} + \text{IV}. 
\end{eqnarray*}

By estimates of $\frac{\partial}{\partial t} p_z(t-s,x-z)$ following from (\ref{dert}) and Theorem \ref{thmq} we get
\begin{equation}
\label{III}
\lim_{h \to 0^+} \text{III} = \int_0^{t/2} \int_{\Rd} \frac{\partial}{\partial t} p_z(t-s,x-z) q(s,z,y) \, dz \, ds.
\end{equation}
We have
\begin{eqnarray*}
\text{IV} &=&
\int_{t/2}^{t} \int_{\Rd} \frac{p_z(t+h-s,x-z)-p_z(t-s,x-z)}{h} (q(s,z,y)-q(s,x,y)) \, dz \, ds\\
&+& \int_{t/2}^{t} \int_{\Rd} \frac{p_z(t+h-s,x-z)-p_z(t-s,x-z)}{h} \, dz q(s,x,y) \, ds\\
&=& \text{V} + \text{VI}. 
\end{eqnarray*}

Note that for $h > 0$, $s \in (t/2,t)$, $\gamma \in (0,\beta)$, $x, y, z \in \Rd$, by Theorem \ref{Holder}  and the estimates of $\frac{\partial}{\partial t} p_z(t-s,x-z)$,  we obtain
\begin{eqnarray*}
&& \left|\frac{p_z(t+h-s,x-z)-p_z(t-s,x-z)}{h}\right| |(q(s,z,y)-q(s,x,y))|\\
&\le& B (|x-z|^{\beta-\gamma} \wedge 1) \left|\frac{\partial}{\partial t} p_z(t+\theta h-s,x-z)\right|\\
&\le& c B (|x-z|^{\beta-\gamma} \wedge 1) (t - s)^{-1} \prod_{i = 1}^d \rho_{\alpha}^0(t-s,x_i-z_i),
\end{eqnarray*}
where $B = B(t,\alpha,d,b_1,b_2,b_3,\beta,\gamma) \in (0,\infty)$ and $\theta = \theta(s,t,h,\alpha,d,b_1,b_2,b_3,\beta,\gamma,x,z) \in (0,1)$. We also have
\begin{eqnarray*}
&& \int_{\Rd} (|x-z|^{\beta-\gamma} \wedge 1) (t - s)^{-1} \prod_{i = 1}^d \rho_{\alpha}^0(t-s,x_i-z_i) \, dz\\
&\le& c \sum_{i = 1}^d \int_{\R} \rho_{0}^{\beta - \gamma}(t-s,x_i-z_i) \, dz_i\\
&\le& c (t - s)^{\frac{\beta-\gamma-\alpha}{\alpha}}.
\end{eqnarray*}
It follows that 
\begin{equation}
\label{V}
\lim_{h \to 0^+} \text{V} = \int_{t/2}^{t} \int_{\Rd} \frac{\partial}{\partial t} p_z(t-s,x-z) (q(s,z,y)-q(s,x,y)) \, dz \, ds.
\end{equation}

Note that for $h > 0$, $s \in (t/2,t)$, $x, z \in \Rd$ we have
\begin{eqnarray*}
&& \frac{1}{h} \left(\int_{\Rd} p_z(t+h-s,x-z) \, dz - \int_{\Rd} p_z(t-s,x-z) \, dz\right)\\
&=& \frac{\partial}{\partial t} \int_{\Rd} p_z(t+\theta h-s,x-z) \, dz\\
&=& \int_{\Rd} \frac{\partial}{\partial t} p_z(t+\theta h-s,x-z) \, dz,
\end{eqnarray*}
where $\theta = \theta(s,t,h,\alpha,d,b_1,b_2,b_3,\beta,\gamma,x) \in (0,1)$.

Using this, (\ref{pyparabolic}) and the definition of $q_0(t,x,y)$ we get
\begin{eqnarray*}
&& \int_{\Rd} \frac{p_z(t+h-s,x-z)-p_z(t-s,x-z)}{h} \, dz\\
&=& \int_{\Rd} \calL^z p_z(t+\theta h-s,\cdot)(x-z) \, dz\\
&=& - \int_{\Rd} q_0(t+\theta h-s,x,z) \, dz + \int_{\Rd} \calL^x p_z(t+\theta h-s,\cdot)(x-z) \, dz.
\end{eqnarray*}
By (\ref{q0}) and (\ref{conv0}),  we have 
$$
\left|\int_{\Rd} q_0(t+\theta h-s,x,z) \, dz \right| \le c (t-s)^{\frac{\beta-\alpha}{\alpha}}.
$$
By Lemma \ref{integralfractional} we get
$$
\left| \int_{\Rd} \calL^x p_z(t+\theta h-s,\cdot)(x-z) \, dz \right| \le c (t-s)^{\frac{\beta-\alpha}{\alpha}}.
$$
It follows that 
\begin{equation}
\label{VI}
\lim_{h \to 0^+} \text{VI} = \int_{t/2}^{t} \int_{\Rd} \frac{\partial}{\partial t} p_z(t-s,x-z) \, dz q(s,x,y) \, ds.
\end{equation}
By (\ref{I1}-\ref{VI}) we obtain 
$$
\lim_{h \to 0^+} \frac{\varphi_y(t+h,x)-\varphi_y(t,x)}{h} = 
q(t,x,y) + \int_0^t \int_{\Rd} \calL^z p_z(t-s,\cdot)(x-z) q(s,z,y) \, dz \, ds.
$$
The proof of the analogous result for $\lim_{h \to 0^-}$ is very similar and it is omitted. 
\end{proof}

\begin{proposition}\label{parabolic1}For all $t \in (0,\infty)$ and $x,y \in \R^d$ we have
\begin{equation*}
\frac{\partial}{\partial t}p^A(t,x,y) = \calL p^A(t,\cdot,y)(x).
\end{equation*}
\end{proposition}

\begin{proof}
By the definition of $q(t,x,y)$ we obtain 
\begin{equation}
\label{qself}
q(t,x,y) = q_0(t,x,y) +
\int_0^t \int_{\Rd} q_0(t-s,x,z) q(s,z,y) \, dz \, ds.
\end{equation}

Using (\ref{pA}), (\ref{pyparabolic}), Lemma \ref{interchange} and the definition of $q_0(t,x,y)$ we obtain
\begin{eqnarray*}
&&\frac{\partial p^A}{\partial t}(t,x,y) = \frac{\partial p_y}{\partial t}(t,x-y) + \frac{\partial \varphi_y}{\partial t}(t,x)\\
&& = \calL^y p_y(t,x-y) + q(t,x,y) +
\int_0^t \int_{\Rd} \calL^z p_z(t-s,\cdot)(x-z) q(s,z,y) \, dz \, ds\\
&& = \calL^x p_y(t,x-y) - q_0(t,x,y) + q(t,x,y) +
\int_0^t \int_{\Rd} \calL^z p_z(t-s,\cdot)(x-z) q(s,z,y) \, dz \, ds.
\end{eqnarray*}
By (\ref{qself}) this is equal to
\begin{eqnarray*}
&& \calL^x p_y(t,x-y) + \int_0^t \int_{\Rd} \left(q_0(t-s,x,z) + \calL^z p_z(t-s,\cdot)(x-z)\right) q(s,z,y) \, dz \, ds\\
&& = \calL^x p_y(t,x-y) + \int_0^t \int_{\Rd} \calL^x p_z(t-s,\cdot)(x-z) q(s,z,y) \, dz \, ds.
\end{eqnarray*}
By (\ref{genvarphi1}) and (\ref{pA}) this is equal to $\calL^x p^A(t,\cdot,y)(x)$, which completes the proof.
\end{proof}

\begin{lemma}
\label{Fubinisemigroup}
For any bounded Borel $f: \Rd \to \R$, $t \in (0,\infty)$ and $x \in \Rd$ we have
$$
\calL (P_t^A f)(x) = \frac{\partial}{\partial t} P_t^A f(x).
$$
\end{lemma}
\begin{proof}
We have
$$
\calL (P_t^A f)(x) = \lim_{\eps \to 0^+} \calL_{\eps} (P_t^A f)(x) 
= \lim_{\eps \to 0^+} \int_{\Rd} \calL_{\eps} p^A(t,\cdot,y)(x) f(y) \, dy.
$$
By Lemma \ref{fractionalheat} this is equal to
\begin{equation}
\label{LPT}
\int_{\Rd} \lim_{\eps \to 0^+}  \calL_{\eps} p^A(t,\cdot,y)(x) f(y) \, dy
= \int_{\Rd} \calL p^A(t,\cdot,y)(x) f(y) \, dy.
\end{equation}
By Proposition \ref{parabolic1} this is equal to
$$
\int_{\Rd} \frac{\partial}{\partial t} p^A(t,x,y) f(y) \, dy 
= \frac{\partial}{\partial t} \int_{\Rd}  p^A(t,x,y) f(y) \, dy. 
$$
\end{proof}

\begin{proposition}\label{ub1}For  $t \in (0,T]$ and $x,y \in \R^d$ we have
\begin{equation*}
 p^A(t,x,y) \le c \prod_{i = 1}^d g_t(x_i - y_i).
\end{equation*}
\end{proposition}

\begin{proof}By  Theorem \ref{thmq}, estimates of $p_z$ and (\ref{conv2}) we obtain 
\begin{eqnarray}
&&|\varphi_y(t,x)| =\left|\int_0^t \int_{\Rd} p_z(t-s,x-z) q(s,z,y) \, dz \, ds \right|\nonumber\\ 
&& \le c \int_0^t \int_{\Rd}  \prod_{i=1}^d \rho_{\alpha}^0(t-s,x_i-z_i)\nonumber\\ 
&& \times \sum_{m=1}^d \left[\prod_{\substack{i=1\\ i \ne m}}^d \rho_{\alpha}^0(s,z_i-y_i)\right]
[ \rho_{0}^{\beta}(s,z_m-y_m) +\rho_{\beta}^0(s,z_m-y_m)] \, dz \, ds \nonumber\\
&& \le c \sum_{m=1}^d \left[\prod_{\substack{i=1 \nonumber\\ i \ne m}}^d \rho_{\alpha}^0(t,x_i-y_i)\right]
\int_0^t \int_{\R} \rho_{\alpha}^0(t-s,x_m-z_m) \\
&& \times [ \rho_{0}^{\beta}(s,z_m-y_m) +\rho_{\beta}^0(s,z_m-y_m)] \, dz \, ds \nonumber\\
&& \le c \sum_{m=1}^d \left[\prod_{\substack{i=1\\ i \ne m}}^d \rho_{\alpha}^0(t,x_i-y_i)\right]
[ \rho_{\alpha+\beta}^{0}(t,x_m-y_m) + \rho_{\alpha}^{\beta}(t,x_m-y_m)]\nonumber\\
&& \le c \left[\prod_{i = 1}^d \rho_{\alpha}^0(t,x_i-y_i)\right]\left[t^{\beta/\alpha} + \sum_{m = 1}^d \left(|x_m-y_m|^{\beta} \wedge 1\right)\right].\label{upperfi}
\end{eqnarray}
Now the conlusion  follows from (\ref{defpA}) and estimates of $p_y$.
\end{proof} 

The following result shows that $\{P_t^A\}$ is a Feller semigroup.
\begin{theorem}
\label{Feller}
We have:

(i) $P_t^A: C_0(\Rd) \to C_0(\Rd)$ for any $ t \in (0,\infty)$,

(ii) $\lim_{t \to 0^+} ||P_t^A f - f||_{\infty} = 0$ for any $f \in C_0(\Rd)$.

(iii) $p^A(t,x,y) \ge 0$ for any $(t,x,y) \in (0,\infty)\times\Rd\times\Rd$,

(iv) $\int_{\Rd} p^A(t,x,y) \, dy = 1$ for any $(t,x) \in (0,\infty)\times\Rd$,

(v) $\int_{\Rd} p^A(t,x,z) p^A(s,z,y) \, dz = p^A(s+t,x,y)$ for any $(s,t,x,y) \in (0,\infty)\times(0,\infty)\times\Rd\times\Rd$.
\end{theorem}
\begin{proof}
(i) follows by the fact that $x \to p^A(t,x,y)$ is continuous and by Proposition \ref{ub1}.

It is shown in the proof of Proposition \ref{ub1} that
$$
|\varphi_y(t,x)| \le c \sum_{m=1}^d \left[\prod_{\substack{i=1\\ i \ne m}}^d \rho_{\alpha}^0(t,x_i-y_i)\right]
[ \rho_{\alpha+\beta}^{0}(t,x_m-y_m) + \rho_{\alpha}^{\beta}(t,x_m-y_m)].
$$
Let $f \in C_0(\Rd)$. It follows that $\lim_{t \to 0^+} \sup_{x \in \Rd} \left| \int_{\Rd} \varphi_y(t,x) f(y) \, dy \right| = 0$ for any $f \in C_0(\Rd)$. It is clear that $\lim_{t \to 0^+} \sup_{x \in \Rd} \left| \int_{\Rd} p_y(t,x-y) f(y)\, dy - f(x)\right| = 0$ for any $f \in C_0(\Rd)$. Hence we obtain (ii).

For any $(t,x) \in (0,T] \times \Rd$ put $u(t,x) = P_t^A f(x)$, $u(0,x) = f(x)$. Note that $u(t,x)$ satisfies the assumptions of Proposition \ref{maximumprinciple}. Indeed, (ii) gives (\ref{mp1}). By Lemma \ref{fractionalheat} we get (\ref{mp2}). By Theorem \ref{mainthm} (iv) and Proposition \ref{gradient1} we obtain (\ref{mp3}). Lemma \ref{Fubinisemigroup} gives (\ref{mp4}). Applying Proposition \ref{maximumprinciple} to $f \in C_c^{\infty}$, $f \le 0$ we obtain (iii). Note that $\tilde{u}(t,x) = -1 + P_t^A 1(x)$, $u(0,x) = 0$ also satisfies the assumptions of Proposition \ref{maximumprinciple}. Using this proposition we get that $P_t^A 1 \equiv 1$ which implies (iv). Fix $s \in (0,T]$, $f \in C_c^{\infty}$, $f \ge 0$ and denote $u_1(t,x) = P_{t+s}^A f(x)$, $u_2(t,x) = P_{t}^A P_s^A f(x)$, $u_1(0,x) = u_2(0,x) = P_s^A f(x)$, $u(t,x) = u_1(t,x) - u_2(t,x)$. By Proposition \ref{maximumprinciple} applied to $u(t,x)$ we get $u_1 \equiv u_2$ which implies (v).
\end{proof}

Using similar ideas as in the proof of (\ref{genvarphi1}) one can easily obtain the following result.
\begin{lemma}
\label{changeoforder}
For any $t \in (0,\infty)$, $x \in \Rd$ and any bounded, H{\"o}lder continuous function $f$ we have
\begin{equation}
\label{changeoforder1}
\calL \left[\int_0^t P_s^A f(\cdot) \, ds\right](x) = \int_0^t \calL P_s^A f(x) \, ds.
\end{equation}
\end{lemma}

\begin{proposition}
\label{martingaleproblem}
For any $t \in (0,\infty)$, $x \in \Rd$ and $f \in C_b^2(\Rd)$ we have
\begin{equation}
\label{martingaleproblem1}
P_t^A f(x) = f(x) + \int_0^t P_s^A \calL f(x) \, ds.
\end{equation}
\end{proposition}
\begin{proof}
Put $u(t,x) = f(x) + \int_0^t P_s^A \calL f(x) \, ds$. By (\ref{changeoforder1}) we get
$$
\calL u(t,x) = \calL f(x) + \int_0^t \calL \left(P_s^A \calL f\right)(x) \, ds.
$$
By Lemma \ref{Fubinisemigroup} this is equal to
$$
\calL f(x) + \int_0^t \frac{\partial}{\partial s} \left(P_s^A \calL f\right)(x) \, ds = P_t^A \calL f(x) = 
\frac{\partial}{\partial t} u(t,x).
$$
It is easy to check that $u(t,x)$ satisfies (\ref{mp1}-\ref{mp3}). Put $\tilde{u}(t,x) = P_t^A f(x)$, $\tilde{u}(0,x) = f(x)$ and $v(t,x) = u(t,x) - \tilde{u}(t,x)$. By the arguments from the proof of Theorem \ref{Feller} we obtain that $\tilde{u}(t,x)$ satisfies (\ref{mp1}-\ref{mp4}). Using Proposition \ref{maximumprinciple} for $v(t,x)$ we get $v \equiv 0$ which implies the assertion of the lemma.
\end{proof}

The next theorem gives that $\calL$ is a generator of the semigroup $\{P_t^A\}$.
\begin{theorem}
\label{generator}
For any $f \in C_b^2(\Rd)$ we have 
$$
\lim_{t \to 0^+} \frac{P_t^A f(x) - f(x)}{t} = \calL f(x), \quad x \in \Rd
$$
and the convergence is uniform.
\end{theorem}
\begin{proof}
By Proposition \ref{martingaleproblem} we have
$$
\lim_{t \to 0^+} \frac{P_t^A f(x) - f(x)}{t} =\lim_{t \to 0^+} \frac{1}{t} \int_0^t P_s^A \calL f(x) \, ds.
$$
By Theorem \ref{Feller} (ii) this is equal to $\calL f(x)$ and the convergence is uniform.
\end{proof}

We are now in a position to provide the proofs of most of the parts of Theorem \ref{mainthm}.
\begin{proof}[proof of Theorem \ref{mainthm} (i), (ii) and the upper bound estimate in (iii)] 
From Theorem \ref{Feller} and  Theorem \ref{generator}  we conclude that there is a Feller process $\tilde{X}_t$ with the transition kernel $p^A(t,x,y)$ and the generator $\calL$. Let  $\p^x, \E^x $ be the distribution and expectation  for the process starting from $x\in \R^d$.
First, note that for any  function $f \in C_b^2(\Rd)$, 
the process
\begin{equation*} \label{martingale0}
M_t^{\tilde{X},f}=f(\tilde{X}_t)-f(\tilde{X}_0)- \int_0^t \calL f(\tilde{X}_s)ds 
\end{equation*}
is a  $(\p^x, \calF_t)  $ martingale, where $\calF_t  $ is a natural filtration. That is  $\p^x$ solves the martingale problem for 
$(\calL, C_b^2(\Rd))$. On the other hand, according to  \cite[Theorem 6.3]{BC2006},  the unique weak solution $X$ to the stochastic equation
(\ref{main}) has the   law which is the unique solution to the  martingale problem for 
$(\calL, C_b^2(\Rd))$. It follows that that $\tilde{X}$ and $X$ have the same law and $p^A(t,x,y)$ is the transition kernel of 
${X}$.

The continuity   of $p^A(t,x,y)$ with respect to all variables follows from Theorem \ref{thmq}.
Positivity  is a consequence of the  lower bound in (\ref{comparability}) which will be proved in the next section. Finally, (ii) follows from Proposition \ref{parabolic1}. The upper bound estimate in (iii) follows from Proposition \ref{ub1}.
\end{proof}
\begin{proof}[proof of Theorem \ref{mainthm} (iv)]
The main tool used in this proof is Theorem \ref{thmq}. 
By Lemma \ref{pyHolder} and Theorem \ref{thmq} we get
\begin{eqnarray}
\nonumber
&& \left|\varphi_y(t,x) - \varphi_y(t,x')\right|\\
\nonumber
&\le& \int_0^t \int_{\Rd} \left|p_z(t-s,x-z) - p_z(t-s,x'-z)\right| \left|q(s,z,y)\right| \, dz \, ds\\
\label{phiy3}
&\le& c \left(A(t,x,y) + A(t,x',y)\right),
\end{eqnarray}
where 
\begin{eqnarray*}
A(t,x,y) &=& |x - x'|^{\gamma} \int_0^t \int_{\Rd} (t-s)^{-\gamma/\alpha} \left[\prod_{i = 1}^d g_{t-s}(x_i-z_i)\right]\\
&& \times s^{d-1} \left[\prod_{i = 1}^d \rho_0^0(s,z_i-y_i)\right] \left[s^{\beta/\alpha} + \sum_{m=1}^d \left(|z_m-y_m|^{\beta} \wedge 1\right)\right] \, dz \, ds.
\end{eqnarray*}
We have
\begin{eqnarray*}
&& A(t,x,y) \le c |x - x'|^{\gamma} \sum_{m=1}^d \left[\prod_{\substack{i=1\\ i \ne m}}^d g_{t}(x_i-y_i)\right]\\
&& \times
\int_0^t \int_{\R} \rho_{\alpha - \gamma}^0(t-s,x_m-z_m) \rho_0^{\beta}(s,z_m-y_m) \, dz_m \, ds\\
&&  + c |x - x'|^{\gamma}  \left[\prod_{i=2}^d g_{t}(x_i-y_i)\right]
\int_0^t \int_{\R} \rho_{\alpha - \gamma}^0(t-s,x_1-z_1) \rho_{\beta}^0(s,z_1-y_1) \, dz_1 \, ds.
\end{eqnarray*}
By (\ref{conv2}) we have 
$$
\int_0^t \int_{\R} \rho^0_{\alpha - \gamma}(t-s,x_m-z_m) \rho_0^{\beta}(s,z_m-y_m) \, dz_m \, ds 
\le \rho_{\alpha - \gamma + \beta}^0(t,x_m-y_m) + \rho_{\alpha - \gamma}^{\beta}(t,x_m-y_m),
$$
and
$$
\int_0^t \int_{\R} \rho_{\alpha - \gamma}^0(t-s,x_1-z_1) \rho_{\beta}^0(s,z_1-y_1) \, dz_1 \, ds 
\le \rho_{\alpha - \gamma + \beta}^0(t,x_1-y_1).
$$
It follows that 
$$
A(t,x,y) \le c |x - x'|^{\gamma} t^{-\gamma/\alpha} \left[\prod_{i = 1}^d g_t(x_i-y_i)\right].
$$
Using this, (\ref{phiy3}), Lemma \ref{pyHolder} and (\ref{pA}) we get the assertion of Theorem \ref{mainthm} (iv).
\end{proof}

\section{Lower bound estimates}

\subsection{L\'evy system}
Let  $\p^x, \E^x $ be the distribution and expectation for the process $X_t$ starting from $x\in \R^d$. By $\calF_t$ we denote a natural filtration. For $x\in \R^d $ and Borel $A\subset \R^d $ we define the jumping measure 
$$J(x, A) = \calA_{\alpha} \sum_{i = 1}^d \int_{A} \otimes_{k\ne i}\delta_{x_k}(dy_k) a^\alpha_{ii}(x)\frac{dy_i}{|y_i-x_i|^{1 + \alpha}},$$
where $\delta_{x_k}$ is a Dirac measure on $\R$ concentrated at $x_k$.  

The purpose of this subsection is to provide arguments for the L\'evy system formula. Namely, we will show that for any $x \in\R^d$
and any non-negative
measurable function  $f$  on $\R_+\times \R^d\times \R^d$
vanishing on
$\{
(s,x,y)\in \R_+\times \R^d\times \R^d; x=y\}$ and $\calF_t$ stopping time $T$, we have
\begin{equation}\label{LS}\E^x \sum_{s\le T}f(s, X_{s-}, X_{s} )= \E^x\int_0^T \int_{\R^d}f(s,X_{s},y)J(X_s, dy)ds.\end{equation}
Since we exactly follow the approach of \cite{CZ} we only briefly  sketch the arguments. 

It is well known that for $f \in C_b^2(\Rd)$,
$$
\calL f(x) = \frac{\calA_{\alpha}}{2} \sum_{i = 1}^d \int_{\R} \left[f(x + a_{ii}(x) w_i e_i) + f(x - a_{ii}(x) w_i e_i) - 2 f(x)\right] \, \frac{dw_i}{|w_i|^{1 + \alpha}}. 
$$

For $y\in \R^d$ we denote $|y|_\infty=  \max_i \{|y_i|\}$ the sup-norm in $\R^d$.
For $x\in \R^d$ and $r>0$ we denote $B(x, r)= \{y\in \R^d, |y-x|_\infty< r\}$. 
Then for $f \in C_b^2(\Rd)$, we can rewrite the formula of the generator as 
$$
\calL f(x) = \lim_{r\searrow 0} \int_{B^c(x, r)} (f(y)-f(x)) J(x, dy). 
$$ 
As it has been already observed in the last section,  for any  function $f \in C_b^2(\Rd)$, the process
\begin{equation*} \label{martingale}
M_t^f=f({X}_t)-f({X}_0)- \int_0^t \calL f({X}_s)ds 
\end{equation*}
is a  $(\p^x, \calF_t)$ martingale. 
Suppose that
$A$
and
$B$
are two bounded closed subsets of
$\R^d$
having a positive distance
from each other. Let $f \in C_b^2(\Rd)$ be such that $f(x)=0, x\in A$ and $f(x)=1, x\in B$. We consider a martingale transform of $M_t^{f}$,
$$N_t^f=\int_0^t{\bf 1}_A(X_{s-})dM_s^{f}. $$
By the Ito formula, if  $X_{s-}\in A$, we have 
$$dM_s^f= f(X_s)-f(X_{s-})-\calL f(X_s)ds= f(X_s)-\calL f(X_s)ds.$$
This implies that
\begin{eqnarray*}N_t^f&=& \sum_{s\le t}{\mathbf 1}_A(X_{s-})f(X_{s})-\int_0^t{\bf 1}_A(X_{s})\calL f(X_s)ds\\
 &=& \sum_{s\le t}{\mathbf 1}_A(X_{s-})f(X_{s})-\int_0^t{\bf 1}_A(X_{s})\int f(y)J(X_s, dy)ds\end{eqnarray*}
Approximating ${\bf 1}_B$ by a decreasing sequence of smooth functions we show that 
$$\sum_{s\le t}{\mathbf 1}_A(X_{s-}){\bf 1}_B(X_{s})-\int_0^t {\bf 1}_A(X_{s})\int_BJ(X_s, dy)ds$$ is a martingale, hence
 $$\E^x \sum_{s\le t}{\mathbf 1}_A(X_{s-}){\bf 1}_B(X_{s})= \E^x\int_0^t {\bf 1}_A(X_{s})\int_BJ(X_s, dy)ds.$$
Using this and a routine measure theoretic argument, we get
$$\E^x \sum_{s\le t}f(X_{s-}, X_{s} )= \E^x\int_0^t \int_{\R^d}f(X_{s},y)J(X_s, dy)ds$$
for any $x \in\R^d$
and any non-negative
measurable function  $f$  on $ \R^d\times \R^d$
vanishing on the diagonal. 

Finally, following the same arguments as in \cite[Appendix A]{CK2008}, we obtain (\ref{LS}).

\

\subsection{Lower bound of $p^A$}
We essentially follow the approach from \cite{CZ}, where an argument relied on certain exit and hitting times estimates was applied, but the singularity of the jumping measure forces us to use an induction argument.     
We start with the near diagonal estimate of the transition kernel. 
\begin{lemma}
\label{diagest} For any $a>0$ there is $c = c(a, d, \alpha, b_1, b_2, b_3, \beta)$ and $0<t_0\le 1$, $t_0 = t_0(a, d, \alpha, b_1, b_2, b_3, \beta)$ such that for $t\le t_0 $ and $x,y\in \R^d$ with $|y-x|_\infty\le at^{1/\alpha}$, 
\begin{equation}
\label{diag} p^A(t,x,y)\ge c t^{-d/\alpha}.  
\end{equation}
\end{lemma}

\begin{proof}

By (\ref{upperfi}), if  $|y-x|_\infty\le at^{1/\alpha}$, we have
$$|\varphi_y(t,x)| \le c \prod_{i = 1}^d \rho_{\alpha}^0(t,x_i-y_i)\left[t^{\beta/\alpha} + \sum_{m = 1}^d \left(|x_m-y_m|^{\beta} \wedge 1\right)\right]\le c_1 t^{-d/\alpha} t^{\beta/\alpha}.$$
Hence,  we can find $t_0\le 1$ such that for $t\le t_0$ and $|y-x|_\infty\le at^{1/\alpha}$ we have 
\begin{eqnarray*}p^A(t,x,y) &&= p_y(t,x-y) + \varphi_y(t,x)\\&&\ge  p_y(t,x-y) - |\varphi_y(t,x)|\ge c_2t^{-d/\alpha}- c_1 t^{-d/\alpha} t^{\beta/\alpha}\\&&\ge c 
t^{-d/\alpha}.\end{eqnarray*}


\end{proof}
Let for a Borel  $D\subset \R^d$,
$$\tau_{D}=\inf\{t>0; X_t\notin D\}\ \text{and} \ T_{D}=\inf\{t>0; X_t\in D\} $$
be the first exit and hitting time of $D$, respectively. 
\begin{lemma}
\label{exit1} There is $c$ such that, for $t\le 1, R>0$, $x\in \R^d$,
 $$\p^x  (\tau_{B(x, R)}\le t  ) \le c\frac t{R^\alpha}.$$
\end{lemma}

\begin{proof} Applying the strong Markow property, we obtain
\begin{eqnarray*} 
                  \p^x  (\tau_{B(x, R)}\le t  )&\le& \p^x  (\tau_{B(x, R)}\le t;\,|X(t)-x|_\infty\le R/8  )+  \p^x  ( |X(t)-x|_\infty\ge R/8  )\\&\le&
             \p^x  (\tau_{B(x, R)}\le t;\, |X(t)-X(\tau_{{B(x, R)}})|_\infty\ge R/8  )\\ &+&  \p^x  ( |X(t)-x|_\infty\ge R/8  ) \\
             &=& \E^x (\tau_{B(x, R)}\le t; P^{X(\tau_{{B(x, R)}})} (|X(t-\tau_{{B(x, R)}})-X(\tau_{{B(x, R)}})|_\infty\ge R/8  ))\\&+&  \p^x  ( |X(t)-x|_\infty\ge R/8 )\\
             &\le& 2\sup_z\sup_{s\le t}\p^z ( |X(s)-z|_\infty\ge R/8)\\
             &\le& c\frac t{R^\alpha}.  
             \end{eqnarray*} 
The last step follows from the upper estimate (\ref{comparability}) of the heat kernel $p^A(t,x,y)$.
\end{proof}

\begin{lemma}
\label{exit2} Let $r>0$ and $x,y\in \R^d$. 
Assume that $|x_1-y_1|\ge 6r$ and $\max_{2\le i\le d}|x_i-y_i|\le r$. Then for $t>0$, 
$$\p^x  (X(t)\in B(y, 4r)   ) \ge c  \frac {rt}{|y_1-x_1|^{1+\alpha}}\p^x (\tau_{B(x, r)}\ge t)\inf_z\p^z(\tau_{B(z, 2r)}> t)$$
\end{lemma}
\begin{proof}
Let $\sigma= T_{B(y, 2r)}$ be the first hitting time. By the strong Markow property
\begin{eqnarray*}  \p^x  (X(t)\in B(y, 4r)   )&\ge&
                  \p^x  (\sigma \le t; \sup_{\sigma\le s\le \sigma+t} |X(s)-X(\sigma)|_\infty\le 2r )\\
             &=& \E^x \left(\sigma\le t; P^{X(\sigma)} \left(\sup_{ s\le t} |X(s)-X(\sigma)|_\infty\le 2r \right)\right)\\
             &\ge& \p^x (\sigma\le t)\inf_z\p^z ( \sup_{s\le t}|X(s)-z|_\infty\le 2r)\\
							&\ge&\inf_z\p^z(\tau_{B(z, 2r)}> t) \p^x  (X(t\wedge\tau_{B(x, r)})\in B(y, 2r) ).  
             \end{eqnarray*} 	
				By the L\'evy system formula (\ref{LS}), we have 
\begin{eqnarray*}
\p^x ( (X(t\wedge\tau_{B(x, r)})\in B(y, 2r) )&=&  \E^x \int_0^{t\wedge\tau_{B(x, r)}} \int_{B(y, 2r)}J(X_s, du)ds.
\end{eqnarray*}
We may assume $x_1<y_1$. Since  $|x_1-y_1| \ge 6r$ and $\max_{2\le i\le d}|x_i-y_i|\le r$, 	
for $z\in B(x, r)$, we have 
$$  \int_{B(y, 2r)}J(z, du)=  \int_{y_1-2r}^{y_1+2r}\frac {a^\alpha_{11}(z)dw_1}{|w_1-z_1|^{1+\alpha}}\ge c \frac {r}{|y_1-x_1|^{1+\alpha}}.$$
Hence,
\begin{eqnarray*}
\p^x ( (X(t\wedge\tau_{B(x, r)})\in B(y, 2r) )&\ge&  c \frac {r}{|y_1-x_1|^{1+\alpha}} \E^x [ t\wedge\tau_{B(x, r)}]\\
&\ge&  c \frac {rt}{|y_1-x_1|^{1+\alpha}} \p^x (\tau_{B(x, r)}\ge t).
\end{eqnarray*} 
\end{proof}


\begin{lemma} 
\label{lower1} There is  $t_0>0$, $t_0 = t_0( d, \alpha, b_1, b_2, b_3, \beta)$, such that for   
 $0<t\le t_0$, $x,y\in \R^d$  
 satisfying   $\max_{2\le i\le d}|x_i-y_i|\le 2t^{1/\alpha}$ we have

$$p^A(t,x,y)\ge c \prod_{i = 1}^d g_t(x_i-y_i).$$
\end{lemma}
\begin{proof}
We pick $t_0>0$ corresponding to $a=12$ in Lemma \ref{diagest}. Due to the near diagonal estimate (\ref{diag}), it is enough to consider $|x_1-y_1|\ge 12t^{1/\alpha}$ and  $\max_{2\le i\le d}|x_i-y_i|\le 2t^{1/\alpha}$.
 Applying Lemma \ref{exit2} with $r=2t^{1/\alpha}$, we obtain
\begin{eqnarray}
p^A(t,x,y)&\ge& \int_{B(y, 4r)} p^A(t_1,x,z) p^A(t_2,z,y)dz \nonumber \\ 
&\ge& \inf_{z\in B(y, 4r)} p^A(t_2,z,y) \p^x ( (X(t_1)\in B(y, 4r))\nonumber\\
&\ge& c\inf_{z\in B(y, 4r)} p^A(t_2,z,y) \frac {rt_1}{|y_1-x_1|^{1+\alpha}} \p^x (\tau_{B(x, r)}\ge t_1) \nonumber\\
&\times& \inf_z\p^z(\tau_{B(z, 2r)}> t_1),\label{LB10}\end{eqnarray}
where $t_i>0$ with $t_1+t_2=t$.

Now, due to Lemma \ref{exit1},  we can pick $0<\lambda<1$, independently of $t$,  such that

$$\inf_z \p^z (\tau_{B(z, r)}\ge \lambda t)\ge 1/2.$$
Moreover, we can select  $\lambda$  so small that $8\le  12 (1-\lambda)^{1/\alpha}$.
Then for $|z-y|_{\infty}\le 4r=8t^{1/\alpha}\le 12 ((1-\lambda)t)^{1/\alpha}$, by Lemma \ref{diagest}, we have 
$$p^A((1-\lambda)t,z,y)\ge c t^{-d/\alpha}.$$
Taking $t_1=\lambda t,   t_2=(1-\lambda) t$ and applying (\ref{LB10}) we arrive at
$$p^A(t,x,y)\ge c t^{-(d-1)/\alpha}\frac {t}{|y_1-x_1|^{1+\alpha}}\ge c \prod_{i = 1}^d g_t(x_i-y_i). $$
\end{proof}
%
%
\begin{proof}[Proof of  the lower bound estimates in Theorem \ref{mainthm} (iii)]
For a natural  $k\le d-1$ we define 
$$V_k(t)=\{(x,y)\in \R^{2d};   \min_{1\le i\le k}|x_i-y_i|\ge t^{1/\alpha} \text{ and} \max_{k+1\le i\le d}|x_i-y_i|\le t^{1/\alpha}\}.$$
We set 
$$V_0(t)=\{(x,y)\in \R^{2d};    \max_{1\le i\le d}|x_i-y_i|\le t^{1/\alpha}\}$$
and
$$V_d(t)=\{(x,y)\in \R^{2d};    \min_{1\le i\le d}|x_i-y_i|\ge t^{1/\alpha}\}.$$

By a renumeration argument it is enough to prove the corresponding lower bound on $V_k(t), k=0,\dots,d$. 
At first, we assume that $t\le t_0 $, where $t_0$ was found in Lemma \ref{lower1}. We have already  proved the lower bound on $V_0(t)$ and $V_1(t)$. 
We show how to extend it to $V_2(t)$.

Thus , we  consider the case $|x_1-y_1|\ge t^{1/\alpha}$, $|x_2-y_2|\ge t^{1/\alpha}$ and $\max_{3\le i\le d}|x_i-y_i|\le t^{1/\alpha}$.

Let $x'=(y_1,x_2,\dots,x_d)$. If $z\in   B(x',t^{1/\alpha}/4)$ then $|x_1-z_1|\ge(3/4) t^{1/\alpha}$, $\max_{2\le i\le d}|x_i-z_i|\le 2t^{1/\alpha}$ and
 $|y_2-z_2|\ge (3/4)t^{1/\alpha}$, $\max_{ i\ne2}|y_i-z_i|\le 2t^{1/\alpha}$. Hence, by   Lemma \ref{lower1},

$$p^A(t,x,z)\ge c t^{-(d-1)/\alpha}\frac {t}{|z_1-x_1|^{1+\alpha}}\ge  c t^{-(d-1)/\alpha}\frac {t}{|y_1-x_1|^{1+\alpha}},$$

$$p^A(t,z,y)\ge c t^{-(d-1)/\alpha}\frac {t}{|z_2-y_2|^{1+\alpha}}\ge  c t^{-(d-1)/\alpha}\frac {t}{|y_2-x_2|^{1+\alpha}}.$$

Finally, 
\begin{eqnarray*}
p^A(2t,x,y)&\ge& \int_{B(x',t^{1/\alpha}/4)} p^A(t,x,z) p^A(t,z,y)dz \\&\ge& c t^{-(d-1)/\alpha}\frac {t}{|y_1-x_1|^{1+\alpha}}t^{-(d-1)/\alpha}\frac {t}{|y_2-x_2|^{1+\alpha}}\int_{B(x',t^{1/\alpha}/4)}dz \\&\ge& c\frac {t}{|y_1-x_1|^{1+\alpha}}\frac {t}{|y_2-x_2|^{1+\alpha}}t^{-(d-2)/\alpha}\\
&\ge& c  \prod_{i = 1}^d g_{2t}(x_i-y_i).
\end{eqnarray*}
This concludes the proof of the lower bound on $V_2(t)$.  In a similar fashion, by induction argument, we show that, if  
$(x,y)\in V_k(t)$ and $t\le t_0$, then 
\begin{eqnarray*}
p^A(t,x,y)&\ge& c\frac {t}{|y_1-x_1|^{1+\alpha}}\times \dots\times\frac {t}{|y_k-x_k|^{1+\alpha}}t^{-(d-k)/\alpha}\\
&\ge& c  \prod_{i = 1}^d g_t(x_i-y_i) \ge c p_0(t,x-y),
\end{eqnarray*}
which ends the proof for the case $t\le t_0$. If $t> t_0$ then we can write $t=nt_0+ s$, with $s<t_0$ and $n\in \N$. Then by already proved lower bound
\begin{eqnarray*}
p^A(t,x,y)&=&\int_{\R^d}\dots\int_{\R^d} p^A(t_0,x,z_1)\dots p^A(t_0,z_{n},z_{n+1})p^A(s,z_{n+1},y) dz_1\dots   dz_{n+1}\\
&\ge & c^{n+1}\int_{\R^d}\dots\int_{\R^d} p_0(t_0,x-z_1) \dots p_0(t_0,z_{n}-z_{n+1})p_0(s,z_{n+1}-y) dz_1\dots   dz_{n+1}\\
&=& c^{n+1} p_0(t,x-y).\end{eqnarray*}
The proof is completed.

\end{proof}

\begin{proof}[proof of Theorem \ref{mainthm} (v)]
The assertion follows from Proposition \ref{gradient1} and the lower bound estimate in Theorem \ref{mainthm} (iii).
\end{proof}



\begin{thebibliography}{99}
\bibliographystyle{plain}

\bibitem{BC2006} R. Bass, Z.-Q. Chen, \emph{Systems of equations driven by stable processes}, 
Probab. Theory Related Fields 134 (2006), no. 2, 175-214.

\bibitem{BGR2014} K. Bogdan, T. Grzywny, M. Ryznar, \emph{Density and tails of unimodal convolution semigroups}, J. Funct. Anal. 266 (2014), 3543-3571.

\bibitem{BJ2007} K. Bogdan, T. Jakubowski, \emph{Estimates of heat kernel of fractional Laplacian perturbed by gradient operators},
Comm. Math. Phys. 271 (1) (2007), 179-198.

\bibitem{BS2007} K. Bogdan, P. Sztonyk, \emph{Estimates of the potential kernel and Harnack's inequality for the anisotropic fractional Laplacian}, Studia Math., 181(2) (2007), 101-123.

\bibitem{BSK2017} K. Bogdan, P. Sztonyk, V. Knopova, \emph{Heat kernel of anisotropic nonlocal operators}, arXiv:1704.03705 (2017).


\bibitem{CH} Z.-Q. Chen, E. Hu, \emph{Heat kernel estimates for $\Delta + \Delta^{\alpha/2}$ under gradient perturbation},
Stochastic Processes Appl. 125 (2015), 2603-2642.

\bibitem{CKK2011} Z.-Q. Chen, P. Kim, T Kumagai, \emph{Global  Heat Kernel Estimates for Symmetric Jump Processes}, Trans. Amer. Math. Soc. 363, no. 9 (2011), 5021-5055.

\bibitem{CKS2012} Z.-Q. Chen, P. Kim, R. Song, \emph{Dirichlet  heat  kernel  estimates  for  fractional Laplacian  with  gradient perturbation}, Ann. Probab. 40 (2012), 2483-2538.

\bibitem{CKS2015} Z.-Q. Chen, P. Kim, R. Song, \emph{Stability of Dirichlet heat kernel estimates for non-local operators under Feynmman-Kac perturbation}, Trans. Amer. Math. Soc. 376 (2015), 5237-5270.

\bibitem{CK2008} Z.-Q. Chen, T. Kumagai, \emph{Heat kernel estimates for jump processes of mixed types on metric measure spaces},
Probab. Theory Relat. Fields 140, no. 1-2 (2008), 277-317.


\bibitem{CZ} Z.-Q. Chen, X. Zhang, \emph{Heat kernels and analyticity of non-symmetric jump diffusion semigroups}, 
Probab. Theory Relat. Fields 165 (2016), no. 1-2, 267-312. 

\bibitem{F1975} A. Friedman, Partial Differential Equations of Parabolic Type, Prentice-Hall, Englewood Cliffs, N.J., (1975).

\bibitem{I2001} Y. Ishikawa, \emph{Density estimate in small time for jump processes with singular L{\'e}vy measures}, Tohoku Math.
J. 53(2) (2001), 183-202. 


\bibitem{JS2012} T. Jakubowski, K. Szczypkowski, \emph{Estimates of gradient perturbation series}, J. Math. Anal. Appl.
389 (2012), 452-460.

\bibitem{JS2010} T. Jakubowski, K. Szczypkowski, \emph{Time-dependent gradient perturbations of fractional Laplacian}, J. Evol.Equ. 10 (2010), 319-339.

\bibitem{KS2015} K. Kaleta, P. Sztonyk, \emph{Estimates of transition densities and their derivatives for jump L{\'e}vy processes}, J. Math.
Anal. Appl. 431(1) (2015), 260-282.

\bibitem{KS2014} K.  Kaleta, P.  Sztonyk, \emph{Small  time  sharp  bounds  for  kernels  of  convolution  semigroups}, J. Anal. Math. 132 (2017), 355-394.

\bibitem{KS2013} K.  Kaleta, P.  Sztonyk, \emph{Upper estimates of transition densities for stable-dominated semigroups}, J. Evol. Equ. 13(3) (2013), 633-650.

\bibitem{KSV2018} P.Kim, R.Song, Z.Vondra{\v{c}}ek, \emph{Heat kernels of non-symmetric jump processes: beyond the stable case}, Potential Analysis, to appear.

\bibitem{K2014} V. Knopova, \emph{Compound  kernel estimates for the transition  probability  density of a L{\'e}vy process in $\R^n$}, Theor. Probability and Math. Statist. 89 (2014), 57-70. 

\bibitem{KK2011} V. Knopova, A. Kulik, \emph{Exact asymptotic for distribution densities of L{\'e}vy functionals}, Electronic Journal of Probability 16 (2011), 1394-1433.

\bibitem{KS2012} V. Knopova, R. Schilling, \emph{Transition  density  estimates  for  a  class  of  L{\'e}vy  and  L{\'e}vy-type  processes}, J. Theoret. Probab. 25(1) (2012), 144-170.

\bibitem{K1989} A. N. Kochubei, \emph{Parabolic pseudodifferential equations, hypersingular integrals and Markov processes},
Math. USSR Izv. 33 (1989), 233-259; translation from Izv. Akad. Nauk SSSR, Ser. Mat. 52 (1988), 909-934.

\bibitem{K2000} V. Kolokoltsov, \emph{Symmetric stable laws and stable-like jump diffusions}, Proc. Lond. Math. Soc. 83 (2000), 725-768.

\bibitem{KR2016} T. Kulczycki, M. Ryznar, \emph{Gradient estimates of harmonic functions and transition densities for L{\'e}vy processes}, Trans. Amer. Math. Soc. 368, no. 1 (2016), 281-318.

\bibitem{LSU1968} O. A. Ladyzenskaja, V. A. Solonnikov, N. N. Ural'ceva, Linear and Quasi-linear Equations of Parabolic Type (translated from the Russian by S.Smith), American Mathematical Society, Providence, (1968).

\bibitem{L1907} E. E. Levi, \emph{Sulle equazioni    lineari    totalmente    ellittiche alle    derivate parziali},
Rend. Circ. Mat. Palermo 24 (1907), 275–317.



\bibitem{M2012} A. Mimica, \emph{Heat kernel upper estimates for symmetric jump processes with small jumps of high intensity},
Potential Anal. 36, no. 2 (2012), 203-222.

\bibitem{P1997a} J. Picard, \emph{Density in small time at accessible points for jump processes}, Stochastic Process. Appl. 67, no.2 (1997), 251-279.

\bibitem{P1997b} J. Picard, \emph{Density in small time for L{\'e}vy processes}, ESAIM Probab. Statist. 1 (1997), 358-389.

\bibitem{P1996} J. Picard, \emph{On the existence of smooth densities for jump processes}, Probab. Theory Related Fields 105 (1996), 481-511.

\bibitem{PSXZ2012} E. Priola, A. Shirikyan, L. Xu, J. Zabczyk, \emph{Exponential ergodicity and regularity for equations 
with L{\'e}vy noise}, Stochastic Process. Appl. 122 (2012), no. 1, 106–133.

\bibitem{PZ2011} E. Priola, J. Zabczyk, \emph{Structural properties of semilinear SPDEs driven by cylindrical stable processes}, Probab. Theory Related Fields 149 (2011), no. 1-2, 97-137.

\bibitem{SSW2012} R. Schilling, P. Sztonyk and J. Wang, \emph{Coupling property and gradient estimates for L{\'e}vy processes via the symbol},
Bernoulli 18 (2012), 1128-1149.

\bibitem{SX2014} X. Sun, Y. Xie, \emph{Ergodicity of stochastic dissipative equations driven by $\alpha$-stable process} Stoch. Anal. Appl. 32 (2014), no. 1, 61-76.

\bibitem{S2010a} P. Sztonyk, \emph{Approximation of stable-dominated semigroups}, Potential Anal. 33(3) (2010), 211-226.

\bibitem{S2017} P. Sztonyk, \emph{Estimates of densities for L{\'e}vy processes with lower intensity of large jumps}, Math. Nachr., 290(1) (2017), 120-141.


\bibitem{S2010b} P. Sztonyk, \emph{Regularity of harmonic functions for anisotropic fractional Laplacians}, Math. Nachr. 283 no.
2 (2010) 289-311.


\bibitem{XZ2014} L. Xie, X. Zhang, \emph{Heat    kernel    estimates for    critical    fractional    diffusion}, Studia Math. 224(3) (2014), 221-263.

\bibitem{WZ2015} L. Wang, X. Zhang, \emph{Harnack Inequalities for SDEs Driven by Cylindrical $\alpha$-Stable Processes}, Potential Anal. 42(3) (2015), 657–669.

\bibitem{Z2013} X. Zhang, \emph{Derivative formulas and gradient estimates for SDEs driven by $\alpha$-stable processes},
Stochastic Process. Appl. 123 (2013) 1213–1228.

\end{thebibliography}
\end{document}